\documentclass[smallextended,envcountsect]{svjour3}

\usepackage{amsmath, amsfonts,amsthm}
\DeclareMathAlphabet      {\mathsl}{OT1}{cmr}{m}{sl}
\usepackage{color}

\usepackage{hyperref}

\usepackage{longtable}

\newtheorem{thm}{Theorem}[section]
\newtheorem{cor}[thm]{Corollary}
\newtheorem{lem}[thm]{Lemma}

\usepackage{mdframed}
\newmdtheoremenv{boxProb}{Problem}
\newmdtheoremenv{boxDef}{Definition}
\newmdtheoremenv{boxCor}{Corollary}
\newmdtheoremenv{compjob}{Computational Job}
\newmdtheoremenv{reqi}{Requirement}

\usepackage{graphicx}
\usepackage{subfigure}          %% Unterabbildungen

\usepackage{placeins}

\usepackage{tabularx}

\usepackage{lmodern,textcomp}

\usepackage{algorithm}
\usepackage{algpseudocode}

\usepackage{enumitem}

\usepackage{xspace}
\newcommand\largeparbreak{\par\bigskip}

\newcommand*\tageq{\refstepcounter{equation}\tag{\theequation}}
\newcommand{\inv}{^{-1}}
\newcommand{\bmu}{\boldsymbol{\mu}}
\newcommand{\htau}{{\hat{\tau}}}
\newcommand{\blambda}{\boldsymbol{\lambda}}
\newcommand{\bxi}{\boldsymbol{\xi}}
\newcommand{\bs}{\mathbf{s}}
\renewcommand{\t}{^{\mathsf{T}}}

\newcommand{\dom}{\operatorname{dom}}

\newcommand{\hp}[1]{^{(#1)}}

\newcommand{\hblambda}{\hat{\blambda}}
\newcommand{\hbmu}{\hat{\bmu}}
\newcommand{\hbz}{\hat{\mathbf{z}}}

\newcommand{\away}[1]{}

\newcommand{\R}{\mathbb{R}}

\newcommand{\N}{\mathbb{N}}

\newcommand{\cB}{\mathcal{B}}
\newcommand{\cN}{\mathcal{N}}

\newcommand{\cF}{\mathcal{F}}

\newcommand{\cU}{\mathcal{U}}

\newcommand{\cO}{\mathcal{O}}

\newcommand{\matleq}{\preceq}

\newcommand{\bA}{\mathbf{A}}

\newcommand{\bG}{\mathbf{G}}

\newcommand{\bV}{\mathbf{V}}

\newcommand{\bQ}{\mathbf{Q}}
\newcommand{\bK}{\mathbf{K}}

\newcommand{\bx}{\mathbf{x}}

\newcommand{\by}{\mathbf{y}}
\newcommand{\bb}{\mathbf{b}}
\newcommand{\br}{\mathbf{r}}
\newcommand{\bc}{\mathbf{c}}
\newcommand{\bg}{\mathbf{g}}

\newcommand{\bd}{\mathbf{d}}

\newcommand{\tbQ}{\tilde{\mathbf{Q}}}
\newcommand{\tbA}{\tilde{\mathbf{A}}}
\newcommand{\tbc}{\tilde{\mathbf{c}}}
\newcommand{\tbb}{\tilde{\mathbf{b}}}

\newcommand{\hbx}{\hat{\mathbf{x}}}
\newcommand{\hbr}{\hat{\mathbf{r}}}
\newcommand{\bz}{\mathbf{z}}
\newcommand{\be}{\mathbf{1}}
\newcommand{\bei}[1]{{\mathbf{e}}}
\newcommand{\bw}{\mathbf{w}}
\newcommand{\bM}{\mathbf{M}}

\newcommand{\bC}{\mathbf{C}}

\newcommand{\bW}{\mathbf{W}}

\newcommand{\bH}{\mathbf{H}}
\newcommand{\bI}{\mathbf{I}}

\newcommand{\bO}{\mathbf{0}}

\newcommand{\tbx}{\tilde{\mathbf{x}}}
\newcommand{\tbu}{\tilde{\mathbf{u}}}
\newcommand{\tbmu}{\tilde{\boldsymbol{\mu}}}
\newcommand{\tblambda}{\tilde{\boldsymbol{\lambda}}}
\newcommand{\tbz}{\tilde{\mathbf{z}}}

\newcommand{\tbr}{\tilde{\mathbf{r}}}
\newcommand{\bu}{\mathbf{u}}
\newcommand{\bv}{\mathbf{v}}

\newcommand{\opdiag}{\mathsl{{diag}}}

\newcommand{\cond}{\mathsl{cond}}

\newcommand{\tol}{{\mathsf{tol}}}

\newcommand{\epsMach}{\varepsilon_{\mathsf{mach}}}

\newcommand{\cbz}{\check{\bz}}
\newcommand{\cbx}{\check{\bx}}
\newcommand{\cblambda}{\check{\blambda}}
\newcommand{\cbmu}{\check{\bmu}}

\newcommand{\boRes}{\chi}
\newcommand{\itersPrimal}{K}
\newcommand{\itersPrimalDual}{M}
\newcommand{\MinFun}{g}

\newcommand{\cGap}{c_{\mathsf{gap}}}

\newcommand{\cDF}{C_{DF}}
\newcommand{\cDFinv}{C_{DFinv}}
\newcommand{\cdF}{C_{\delta{}F}}
\newcommand{\cdDF}{C_{\delta{}DF}}

\newcommand{\regQP}{\textsf{regQP}}

\newcommand{\cmu}{\check{\mu}}

\newcommand{\Cx}{\sqrt{n}}

\newcommand{\thbz}{\hat{\tbz}}

\newcommand{\boxQP}{\textsf{boxQP}}
\newcommand{\prQP}{\textsf{primalQP}}
\newcommand{\duQP}{\textsf{dualQP}}
\newcommand{\stQP}{\textsf{standardQP}}

\newcommand{\ConstantsRef}{Appendix~C\xspace}

\usepackage{todonotes}

\title{Stable interior point method for convex quadratic programming with strict error bounds}

\author{Martin Neuenhofen and Stefania Bellavia}

\institute{Martin Neuenhofen \at
	  	Department of Computer Science\\
	  	University of British Columbia\\
	  	Vancouver, Canada\\
		\url{www.MartinNeuenhofen.de}
	\and
		Stefania Bellavia \at
      	Dipartimento di Ingegneria Industriale\\
		Universit\`{a} di Firenze \\
		Florence, Italy\\
		\url{http://www2.de.unifi.it/anum/bellavia/}
}

\titlerunning{Stable IPM for convex QP with strict error bounds}
\authorrunning{M.~Neuenhofen \& S.~Bellavia}

\newcommand\suppressqed{\renewcommand{\qedsymbol}{}}
\begin{document}

\maketitle

\begin{abstract}
We present a short step interior point method for solving a class of nonlinear programming problems with quadratic objective function. Convex quadratic programming problems {can be reformulated as problems} in this class. The method is shown to have weak polynomial time complexity. {A complete proof of the numerical stability of the method is provided.} No requirements on feasibility, row-rank of the constraint Jacobian, {strict complementarity,} or conditioning of the problem are made. Infeasible problems are solved to an optimal interior least-squares solution.
\end{abstract}

\section{Introduction}
This paper is concerned with the numerical solution of the following nonlinear programming problem with quadratic objective function:
\begin{equation}
 \tag{\boxQP}
  \begin{aligned}
%	\begin{align}
		\operatornamewithlimits{min}_{\bx \in \R^n}& & & q(\bx):=\frac{1}{2}\cdot\bx\t\cdot\bQ\cdot\bx+\bc\t\cdot\bx\\
		\text{subject to}& & &\|\bA\cdot\bx-\bb\|_2=\boRes\,,\quad \|\bx\|_\infty \leq 1
	%\end{align}
 \end{aligned}
\label{eqn:mainQP}%
	%\todo{tageq with \boxQP.. Replace eqrefs with \boxQP}
\end{equation}%
where $\bQ \in \R^{n \times n}$ is symmetric positive semi-definite, $\bc\in\R^n$, $\bA\in\R^{m \times n}$, $\bb\in\R^n$,
 and  $\boRes$ defined as
\begin{align*}
	\boRes:= \operatornamewithlimits{min}_{\bxi \in \overline{\Omega}}\big\lbrace \,\|\bA\cdot\bxi-\bb\|_2\, \big\rbrace\,\quad\quad	
	\mbox{ with    }\quad\Omega:=  \big\lbrace \bxi \in \R^n \ \vert \ \|\bxi\|_\infty<1 \big\rbrace\,.
\end{align*}%\todo{tageq with (\stQP). Replace eqrefs with (\stQP)}
The dimension $m$ can be either smaller, equal or larger than $n$.
This problem is bounded and feasible by construction as  $\boRes$ is computed minimizing  $\|\bA\cdot\bxi-\bb\|_2$ in $\overline{\Omega}$. The problem reduces to a convex quadratic program whenever $\boRes=0$.
Note that we do not need to know $\boRes$ in advance. An approximation of $\boRes$ is a by-product of our method. In case the problem is not feasible, that is $\boRes>0$, the solution of \eqref{eqn:mainQP} in an optimal interior least-squares solution. We write $\bx^\star$ for an arbitrary minimizer of \eqref{eqn:mainQP}.

The motivation for considering this problem is that a convex  quadratic program (CQP) in standard form can be reformulated as \eqref{eqn:mainQP} provided that an upper bound for the infinity norm of a solution is known.
Let us consider the following CQP in standard form \cite[eq. (1.22)]{Wright}
\begin{equation}
	\tag{\stQP}
\begin{aligned}
	\operatornamewithlimits{min}_{\tbx\in\R^n}& & &\tilde{q}(\tbx):=\frac{1}{2} \cdot \tbx\t\cdot\tilde{\bQ}\cdot\tbx+\tilde{\bc}\t\cdot\tbx\\
	\text{subject to}& & &\tbA \cdot \tbx = \tilde{\bb}\,,\quad \tbx\geq \bO
\end{aligned}\label{eq:stQP}
\end{equation}
with a minimum-norm solution $\tbx^\star$,  $\tilde{\bQ} \in \R^{n \times n}$ symmetric positive semi-definite, $\tilde\bc\in\R^n$,
$\tilde\bA\in\R^{m \times n}$, $\tilde\bb\in\R^n$. If an upper bound $\pi$ for $\|\tbx^\star\|_\infty$ is known then this problem can be cast to form \eqref{eqn:mainQP} by employing the substitution  $\tbx=0.5\cdot\pi\cdot(\bx+\be)$. In fact,
the above CQP  can be reformulated as \eqref{eqn:mainQP} where $\bA=\pi/2\cdot\tbA$, $\bb=\tbb-\pi/2\cdot\tbA\cdot\be$, $\bQ=0.25\cdot\pi^2\cdot\tbQ$,
$\bc=0.5 \cdot \pi\cdot \tbc + 0.25 \cdot \pi^2 \cdot \tbQ\cdot\be$, where $\be$ denotes the vector of all ones.

Figure~\ref{fig:sandardcqptrafoboxcqp} shows the relation between \eqref{eq:stQP} and \eqref{eqn:mainQP}. For feasible problems follows $\boRes=0$ {and the problem becomes equivalent to a convex quadratic program.}

We also underline that a sharp bound $ \pi$ is not needed and it is enough to use a large $\pi$. In case it is not available, one could successively attempt solving \eqref{eqn:mainQP} for a geometrically growing sequence of trial values for $\pi$. This requires $\cO\big(\log(\|\tbx^\star\|_\infty)\big)$ attempts. The guess for $\pi$ is large enough when the trial solution stays away from the right borders, e.g. $\bx^\star < 0.9\cdot\be$.

\begin{figure}
	\centering
	\includegraphics[width=0.9\linewidth]{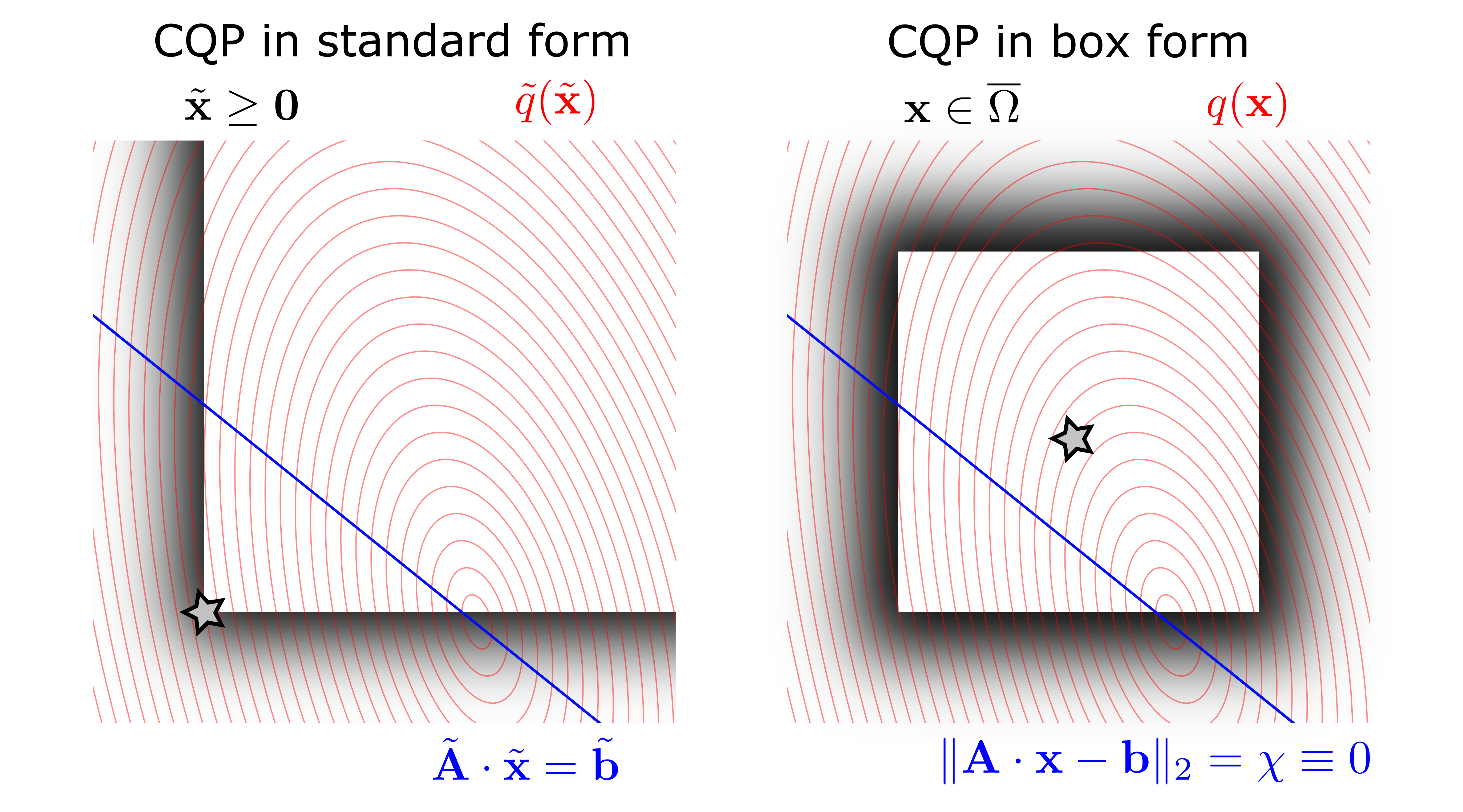}
	\caption{Illustration of transformation of CQP in standard form in two dimensions. The problem is rescaled into a box. Stars mark the positions of the null-vectors in the respective coordinate space. For a suitable rescaling the solutions of both problems coinside. Since the depicted problem has feasible points, it follows $\boRes=0$.}
	\label{fig:sandardcqptrafoboxcqp}
\end{figure}

In this paper we develop an interior point method for problem \eqref{eqn:mainQP}. It consists of an initialization phase, where a point belonging to a small neighbourhood of the central path is generated, and of a short step path-following phase, where the complementarity is iteratively reduced. The involved linear equation systems are regularized, but despite this  the algorithm recovers a solution of the original problem.
In this sense it shares some similarities with the regularized interior point method in \cite{FriedlanderOrban}.
In fact, our method, given a prescribed accuracy   $0<\tol\in\R$,  computes a vector $\bx \in \R^n$ satisfying
\begin{subequations}
	\begin{align}
		\bx &\in \Omega\\
		q(\bx) &\leq q(\bx^\star) + \tol\,,\\
		\|\bA\cdot\bx-\bb\|_2 &\leq \boRes + \tol\,.
	\end{align}\label{eqn:SolutionConditions}
\end{subequations}

For interior point methods it is convenient to define an \textit{input length number} $L$, with respect to which the method achieves weakly polynomial time complexity. $L$ is a positive number with a magnitude in the order of the bits that are needed to store the problem instance. We present an algorithm that can solve every real-valued problem in weakly polynomial time, and real-valued problems cannot be represented in a finite number of bits. This is why instead of an input length we define a \textit{problem factor} $L$ of \eqref{eqn:mainQP} in the following way:
\begin{align}
	L:=\log\big(1+\|\bQ\|+\|\bc\|+\|\bA\|+\|\bb\|\big)+\log(n+m)-\log(\tol)\,,\label{eqn:InputLength}
\end{align}
where $\|\cdot\|$ is an arbitrary vector norm. $L$ is the weak factor with regard to which our algorithm has weakly polynomial time complexity. Since only consisting of logarithms, this factor grows slowly and is thus of practically reasonable size. The exact solution of real-valued convex quadratic problems is NP-hard \cite{478584}. Our algorithm avoids this hardness by not solving exactly but only to a tolerance $\tol>0$. Clearly, the tolerance must appear in $L$ since it brings the problem parametrically close to NP. Since $L$ grows logarithmically in $\tol$, high accuracies can be achieved with our algorithm at reasonable time complexity.

Our method achieves the following desirable goals:
\begin{enumerate}
\item No requirements or assumptions are made on the regularity, rank or conditioning of $\bA,\bQ$, on strict complementarity at $\bx^\star$, on feasibility or on anything else. This method  is free of any assumptions. We underline that typically, assumptions are made on full row-rank of $\tbA$ \cite[p.~31]{Wright}. This assumption can be forced by orthogonalizing the rows of $\tbA$ but this is not feasible when $\tbA$ is large and sparse or unavailable, i.e. it is not compatible with a  matrix-free regime.

\item The input length $L$ does only consist of logarithms of the input data and dimension. Thus, the iteration count of the method is mildly affected by potentially bad scaling of the problem.

\item The feasible initialization of our method does not require augmentation of $\bA,\bQ$. This is desirable for sparsity and simplicity. In other interior point methods, the issue of finding a feasible initial guess is generally treated in either of three ways:  infeasible methods \cite{GondzioIIfIPM,YYMY,Wright,IPM25ylater,POTRA2000281}, a two-phase method \cite[Sec.\,11.4]{Boyd}, or a "big-M"-method \cite{AdlerMonteiro,Gondzio_bigM}.

Infeasible methods assume $\|\tbx_0-\tbx^\star\|_\infty \in \cO(1)$ because this term  appears in the denominator of the line-search step-length, cf. \cite{YYMY} equations (28), (32), (48) and the definition of $\rho,x_0,C_1,C_2,C_3,C_4,C$ in their notation. Note that $\|\tbx_0-\tbx^\star\|_\infty \in \cO(1)$ would require $\pi \in \cO(1)$ unless $\tbx_0$ is a very accurate approximation to $\tbx^\star$. Large values of $\pi$ would quickly blow up the worst-case iteration count, which is a weakness of these methods.

The two-phase method approach exchanges the optimality conditions during the course of iterating. The theory of this approach can be inaccurate when the solution of phase I does not solve the different equations of phase II. The "big-M"-approach requires a user-specified bound for $\pi$, just as we do. (In \cite{AdlerMonteiro} the problem is integer and thus $\pi \leq 4^L$ holds for a different definition of $L$, that is impractical to compute when $\bA$ is large and sparse.) We believe it is best practice to consider \eqref{eqn:mainQP}, as we do.

 Our approach of presuming a bound is further justified in the fact that it is impossible to determine boundedness of $\|\tbx^\star\|_\infty$ for \eqref{eq:stQP} in a numerically robust way. To show this, consider in $n=2$ dimensions the problem given by $\tbQ=\bO,\ \tbc=(-1,0)\t,\ \tbA=[\epsilon,1],\ \tbb=1$ with a small number $\epsilon>0$. The problem appears to be well-scaled, yet the solution is $\tbx^\star=(1/\varepsilon,0)\t$ and its norm is not well-posed with respect to small perturbations of the problem data. Thus, when we consider solving \eqref{eqn:mainQP} instead of \eqref{eq:stQP} then we circumvent the numerical issue of determining boundedness, since by construction we consider a bounded problem. We are aware that for unbounded programs the number of attempts for guessing $\pi$ is infinite. But we also believe that the user is not interested in solutions that exceed a particular value of $\pi$, e.g. $\pi=10^{100}$.

\item The condition numbers of the arising linear systems remain bounded during all iterations of the algorithm. This also holds when the solution is approached -- even for degenerate problems, which is recognized in \cite{Wright,GreifOrban} to be a difficult case. Our method is capable of this because of the regularization of the linear systems and because all iterates stay sufficiently interior and bounded during all iterations.

\item The method has the best known worst-case iteration-complexity, namely it requires $\cO(L \cdot \sqrt{n})$ iterations, cf. \cite{IPM25ylater,AdlerMonteiro}. The time complexity of our method consists of solving $\cO(L \cdot \sqrt{n})$ well-conditioned linear systems of dimension $\leq 3 \cdot n + m$.

\item Our method can also handle infeasible problems. Using the common primal-dual formulation, infeasible problems are known to result in blow-up of the dual variables, cf. \cite[p.~177,~ll.~21--22]{Wright} and the reference \cite{BlowUp} therein. According to \cite[Chap. 9]{Wright} and \cite[p. 411]{Nocedal} there are approaches of exploiting this blow-up for infeasibility detection. However, problems may be infeasible due to floating-point representation. Nevertheless we want to solve them in a robust and meaningful way, i.e. as problem \eqref{eqn:mainQP}.

A further issue for an infeasible CQP can be with the initialization: All interior point methods are iterative, and they first need to construct an initial point that lives in the neighborhood of the central path. If now the path is not well-defined or even non-existent for infeasible problems then the initialization will fail, causing a catastrophic breakdown. To avoid this, a suitable definition of a central path is needed that is also meaningful for infeasible problems. Since we work with the reformulated problem \eqref{eqn:mainQP} we have a suitable definition of the central path that works also in the infeasible case (i.e. $\boRes>0$).

\item In our analysis we considered the effects of numerical rounding errors. We show that our method is numerically stable.

Proving stability appears difficult because the linear systems in the Newton iteration can be badly conditioned. Substantial efforts have been undertaken towards a stability result in earlier works, cf. \cite{FriedlanderOrban,GreifOrban,BeneSimon,Wright95simax}. These are not exhaustive and live from strong assumptions such as regularity, strict complementarity, feasibility and others. However, the sole analysis of the linear equation systems is not sufficient for a proof of the overall stability of the interior point algorithm. Further questions relating to maintenance of strict interiorness and solution accuracy of the final result need to be addressed.

To address these issues we make use of enveloped neighborhoods of the central path and we enrich the path-following procedure with two further Newton steps at each path following iteration. The role of these two steps is to compensate for numerical rounding errors. This way, we make sure that a new point does never move too far away from the central path. We emphasize that there is a trade-off between the design of a cheap iteration and a numerically stable iteration. The two additional Newton steps enable us to provide a complete proof of numerical stability without assuming strict complementarity at the solution, that is a common assumption in papers dealing with numerical stability of interior point methods.

\item Most numerical algorithms do only compute a solution that yields small residual norms for the KKT equations. This holds some uncertainty because small residuals do not imply at all that the numerical solution is accurate in terms of optimality gap and feasibility residual (also with respect to $\boRes$). In contrast to that, our algorithm finds a solution that satisfies \eqref{eqn:SolutionConditions}. Moreover, the tolerances of eventually found optimality gap and feasibility residual can be bounded a-priori by the user by specifying the value of $\tol$. Thus, the algorithm will never return an unsatisfactory solution.
\end{enumerate}

The focus of this paper is on proving the theoretical properties of our method, namely the analysis of worst case time complexity, correctness/ solution accuracy, and numerical stability subject to rounding errors. In addition to that, also the following properties are highly relevant from a practical point of view: linear systems with a condition number that does not blow up, the possibility to detect infeasibility of problems in a quantitative measure $\boRes$, and the ability to solve also problems that are rank-deficient and degenerate. Our analysis in this paper paths also the way for designing long step variants of the method, the development of which is currently under way.

\largeparbreak

We describe the structure of the paper. In Section~2 we describe the numerical method. In Section~3 we prove the correctness of the method. Section~4 gives the stability analysis. Eventually we draw a conclusion.
In Appendices A and B we report proofs of the  results concerning correctness and stability of the method.

\subsection{Notation}
We write $\be$ for a vector of all ones. $\bO$ is a null-matrix/vector and $\bI$ is the identity. The dimension is either clear or mentioned as sub-indices. We write $u^{[j]}$, $1 \leq j \leq d$ for the $j$th component of a vector $\bu \in \R^d$. We write $1/\bu$ for the vector $\opdiag(\bu)\inv \cdot \be\,,$ $\bu \cdot \bv$ for $\opdiag(\bu) \cdot \bv$ and $\bu^2$ for $\bu\cdot\bu$\,, where $\opdiag(\bu)$ denotes the diagonal matrix whose diagonal entries are given by the components of  $\bu $.
The scalar-product instead is $\bu\t\cdot\bv$. We define $\cB_R(\bu):= \lbrace \tbu \in \R^d \ \vert \ \|\bu-\tbu\|_2\leq R\rbrace$, the sphere with radius $R$ around $\bu$. In analogy we use spherical envelopes \mbox{$\cB_R(\cU) := \lbrace \bu \in \R^d \ : \ \text{dist}(\bu,\cU) \leq R \rbrace$} for a set $\cU \subset \R^d$ with radius $R$, where $\text{dist}$ uses the Euclidean metric.

Throughout the paper we use a list of scalar \textit{method parameters} that can be directly computed from the given problem data $\bQ,\bc,\bA,\bb,\tol$. We present them here in one place in Appendix~C. Details on their computation are found also in the appendix. For each method parameter $s \in \R$ it holds $|\log(s)|\in\cO(L)$. This is crucial for our complexity analysis. The method parameters are literal numbers whose values are determined from the moment when the problem instance is given to the solver. We will frequently use various of these numbers in the algorithm, our theorems, equations and our analysis. For the sake of a succinct presentation and in order to avoid redundancies, we will not refer to \ConstantsRef each time a method parameter is used.

\section{The Algorithm}

In this section we describe our short step interior point method for problem \eqref{eqn:mainQP}. Given a prescribed tolerance $\tol>0$, the method computes a numerical solution satisfying \eqref{eqn:SolutionConditions}.

\paragraph{Functions}
To state the algorithm, we first need to introduce the following auxiliary function that we refer to as \textit{primal function}:
\begin{align*}
f(\bx) := &\frac{1}{\tau_A} \cdot \Bigg(q(\bx) + \frac{\omega}{2} \cdot \|\bx\|_2^2 + \frac{1}{2 \cdot \omega} \cdot \|\bA \cdot \bx - \bb\|_2^2 \Bigg)\\
&- \sum_{j=1}^n \left(\log\big(1+x^{[j]}\big)+\log\big(1-x^{[j]}\big)\right) \tageq,\label{eqn:f_primal}
\end{align*}
$\tau_A$ and $\omega$ are method parameters defined in \ConstantsRef. $f$ is a regularized penalty-barrier function. It has the following properties:
\begin{itemize}
	\item $f$ is self-concordant.
	
	\item $f$ has a unique minimizer $\bx_{\infty}$. $f$ is strictly convex since $q$ is convex and due to the additional term $\frac{\omega}{2} \cdot \|\bx\|_2^2$.
	
	\item An accurate guess for $\bx_{\infty}$ is readily available. This is because $\tau_A$ has a large value. If $\tau_A=+\infty$ then $f$ would only be the log-part, whose unique minimizer is $\bx_\infty = \bO$. Now since $\tau_A$ is large it follows that $\bx=\bO$ is very close to the minimizer of $f$.
\end{itemize}
Due to the latter two conditions, Newton's method for minimization finds an accurate minimizer of $f$ in $K$ iterations from the initial guess $\bx_0 = \bO$\,, where $K$ is a method parameter. Typically it holds $\itersPrimal \leq 10$\,, cf. \cite[p. 489]{Boyd}.

Let us also introduce the parametric function $F_\tau : \R^{N} \rightarrow \R^N$ defined as
\begin{align}
F_\tau(\bz) := \begin{pmatrix}
\bQ \cdot \bx + \omega \cdot \bx + \bc - \bA\t \cdot \blambda - \bmu_L + \bmu_R\\
\bA \cdot \bx - \bb + \omega \cdot \blambda\\
\bmu_L \cdot (\be+\bx) - \tau \cdot \be\\
\bmu_R \cdot (\be-\bx) - \tau \cdot \be
\end{pmatrix}\,,\label{eqn:F_KKT}
\end{align}
parametric in $\tau>0$, where $N:=3 \cdot n + m$ and $\bz \equiv (\bx\t,\blambda\t,\bmu_L\t,\bmu_R\t)\t \in \R^{N}$.
We will refer to $F_\tau$ as \textit{optimality function}. It is related to the KKT function used in common primal-dual interior point methods like \cite{IPM25ylater}, but it has some significant differences. It can be derived from the optimality condition $\nabla \varphi_\tau(\bx) = \bO$ for the following parametric penalty-barrier function
\begin{align*}
\varphi_\tau(\bx) = & q(\bx) + \frac{\omega}{2} \cdot \|\bx\|_2^2 + \frac{1}{2 \cdot \omega} \cdot \|\bA \cdot \bx - \bb\|_2^2 \\
&- \tau \cdot \sum_{j=1}^n \left(\log\big(1+x^{[j]}\big)+\log\big(1-x^{[j]}\big)\right)\,.
\end{align*}
In $\nabla \varphi_\tau(\bx) = \bO$ we substituted $\blambda = \frac{-1}{\omega} \cdot (\bA \cdot \bx - \bb)$, $\bmu_L = \tau / (\bx+\be)$ and $\bmu_R = \tau / (\bx-\be)$. $F_\tau$ has the following Jacobian.
\begin{align}
DF(\bz) = \begin{bmatrix}
\bQ + \omega \cdot \bI 	& -\bA\t  			& -\bI 				& \bI				\\
\bA						& \omega \cdot \bI 	& \bO 				& \bO 				\\
\opdiag(\bmu_L)			& \bO 				& \opdiag(\be+\bx) 	& \bO 				\\
-\opdiag(\bmu_R) 		& \bO				& \bO 				& \opdiag(\be-\bx)	
\end{bmatrix}\label{eqn:DF_KKT}
\end{align}
Notice that we dropped the foot-index $\tau$ from $DF_\tau$ since it does not depend on $\tau$.

Solving a minimization problem with $\varphi_\tau$ for a decreasing sequence of $\tau$ allows to achieve the following goals:
\begin{itemize}
	
	\item Since $\varphi_\tau$ is strictly convex, it has a unique global minimizer for each value $\tau>0$.
	
	\item The penalty term $\frac{1}{2 \cdot \omega} \cdot \|\bA \cdot \bx - \bb\|_2^2$ is very large. Thus, minimization of $\varphi_\tau$ will approximately (but in the right numerical scales of accuracy) imply to minimize $q(\bx)$ subject to minimality of $\|\bA \cdot \bx - \bb\|_2$.
	
	\item A sequence of sufficiently accurate minimizers of $\varphi_\tau$ for a decreasing sequence of values $\tau$ can be computed in a primal-dual path-following framework in an efficient computational complexity.
	
	\item Finally, the convex term $\frac{\omega}{2} \cdot \|\bx\|_2^2$ makes sure that the conditioning of $DF$ is bounded at all iterates $\bz$ that appear during the execution of the algorithm.
	
\end{itemize}

\paragraph{Spaces} Interior point methods are based on the concept of a central path \cite{Wright}. We define the following $\tau$-parametric spaces, that define a neighborhood of the central path.
\begin{align*}
\cN(\tau) := \Big\lbrace \bz \in \R^N\,\Big\vert \quad & F_\tau(\bz)=(\bO_n,\bO_m,\br\hp{3},\br\hp{4})\,,\\
& \|(\br\hp{3},\br\hp{4})\|_2\leq \theta \cdot\tau\,,\\
& \|\bx\|_\infty<1\,,\ \bmu_L,\bmu_R>\bO \quad\quad\Big\rbrace\tageq\,.
\end{align*}
$\theta$ is a method parameter that determines the width of the neighborhood from the central path. If $\theta=0$ would hold then the spaces would only consist of those points that live precisely on the central path itself. Besides $\cN(\tau)$ we also define the spaces $\cN_h(\tau)$. These differ from $\cN(\tau)$ only in the detail that they use half the width for the neighborhood:
\begin{align*}
\cN_h(\tau) := \Big\lbrace \bz \in \R^N\,\Big\vert \quad & F_\tau(\bz)=(\bO_n,\bO_m,\br\hp{3},\br\hp{4})\,,\\
& \|(\br\hp{3},\br\hp{4})\|_2\leq 0.5 \cdot \theta \cdot\tau\,,\\
&\|\bx\|_\infty<1\,,\ \bmu_L,\bmu_R>\bO \quad\quad\Big\rbrace\tageq\,.
\end{align*}

When it comes to numerical roundoff, neither of the above spaces is practically useful. This is because there may be no digitally representable number for $\bz$ that solves $F_\tau(\bz)=(\bO_n,\bO_m,\br\hp{3},\br\hp{4})$\,. This is why we also introduce \textit{enveloped neighborhood spaces}. These are the above spaces $\cN,\cN_h$ plus an envelope $\nu\geq 0$:
\begin{align*}
	\cN(\tau,\nu) &:= \Big\lbrace \tbz \in \R^N \ \Big\vert \ \exists \bz\in \cN(\tau) \ : \ \|\tbz - \bz\|_2 \leq \nu \Big\rbrace \equiv \cB_\nu\Big(\,\cN(\tau)\,\Big)\\
	\cN_h(\tau,\nu) &:= \Big\lbrace \tbz \in \R^N \ \Big\vert \ \exists \bz\in \cN_h(\tau) \ : \ \|\tbz - \bz\|_2 \leq \nu \Big\rbrace \equiv \cB_\nu\Big(\,\cN_h(\tau)\,\Big)
\end{align*}

\paragraph{Working strategy of the algorithm}
Now we have all the ingredients at hand to describe the proposed algorithm. It is stated in Alg.\,\ref{Algo:IPM}. The algorithm consists of three sections: Initialization, path-following, and termination. Corresponding to the crucial steps of the algorithm we added some comments explaining the role of these steps. We will refer to these comments in the theoretical analysis of the method.

The initialization has two parts: In the first part we start with a primal Newton iteration to find an approximate minimizer $\bx_K$ of $f$. In line 9 we find that this approximate minimizer satisfies an error-bound. In the second part a vector $\cbz$ is computed from the approximate minimizer. We find that $\cbz$ lives in a special neighborhood space. In the third part $\cbz$ is refined with a modified primal-dual Newton step. From this Newton step we obtain $\bz \in \cN(\tau_A,\nu_0)$, which is a suitable vector to begin the path-following with.

The path-following is a loop that computes a vector $\bz \in \cN(\tau_E,\nu_0)$ by iteratively computing refinements for $\bz$ which live in spaces $\cN(\tau,\nu)$ for geometrically decreasing values of $\tau$ and bounded values for $\nu$. The loop consists of three parts: First, a path-step is performed in order to update the primal-dual iterate so that $\tau$ decreases. Afterwards, a centrality step is performed to move from the neighborhood $\cN$ into $\cN_h$. Finally, an error-reset step is computed. This decreases the value of $\nu$ so that it always remains bounded. In particular, while the first two steps may lead to growth of $\nu$, the third step reduces it below a method parameter $\nu_0$\,.

We emphasize that our algorithm is designed so that it also works when there are numerical rounding errors. If there were no rounding errors then it would hold $\nu_0=\nu_1=0$ and the computation of $\Delta\bz_2,\Delta\bz_3$ could be replaced by $\Delta\bz_2=\bO, \Delta\bz_3=\bO$. That is, when numerical stability is not a concern then one can opt for solving only one linear system per path-following iteration, as in classical methods.

In the termination we have a primal-dual iterate that lives in a special neighborhood space. We will give a result which says that members of these spaces have a component $\bx$ which satisfies \eqref{eqn:SolutionConditions}.

\paragraph{Complexity}
Our method has to perform $\itersPrimal$ iterations for the initialization phase and $\itersPrimalDual$ iterations of the path-following phase, where $\itersPrimal$ and $\itersPrimalDual$ are {method parameters}.
In absence of rounding errors one Newton system has to be solved at each path-following iteration.

We find that the time complexity of the method is bounded by solving $\itersPrimal+\cO(\itersPrimalDual)$ linear systems of dimension $\leq N$\,. We find that the following complexities hold:
\begin{align*}
\itersPrimal &= \Big\lceil \log_2\big( \underbrace{1 + \log_2(\overbrace{C_{Hf}/\rho}^{>1})}_{\leq  C_{Hf}/\rho} \big) \Big\rceil \in \cO(L)\\
\itersPrimalDual &= \bigg\lceil \overbrace{\big( \log(\tau_A) - \log(\tau_E) \big)}^{> 0} \cdot \overbrace{\frac{-1}{\log(1-\beta/\sqrt{2\cdot n})}}^{>0}\bigg\rceil\\
& \le  \bigg\lceil \underbrace{\big( \log(\tau_A) - \log(\tau_E) \big)}_{\in \cO(L)} \cdot  1/\beta \cdot \sqrt{2 \cdot n} \bigg\rceil\in \cO(L \cdot \sqrt{n})
\end{align*}
where all the constants involved are given in \ConstantsRef and their logarithms are in the order of $L$. In the bound for $\itersPrimal$ we used $\log(C_{Hf}/\rho) \in \cO( L )$\,.
In result, the algorithm's iteration complexity is
$$\cO\big(L \cdot \sqrt{n}\big)\,.$$
The required memory depends on how the linear equation systems are solved. If matrix-free methods are used then only the vectors $\bx,\bz$ of size $\cO(N)$ must be kept in memory.

\begin{algorithm}
\normalsize{
\begin{algorithmic}[1]
\Procedure{Solver}{$\bQ,\bc,\bA,\bb,\tol$}
	\State Compute all the method parameters from Appendix C.\vspace{2mm}
	\State \textit{// - - - Initialization - - -}
	\State $\bx_0 := \bO \in \R^{n}$
	\For{$k=1,...,\itersPrimal$}
		\State Solve $\nabla^2f(\bx_{k-1}) \cdot \Delta\bx_k = -\nabla f(\bx_{k-1})$ for $\Delta\bx_k$\label{AlgLine:Newton1}
		\State $\bx_k := \bx_{k-1} + \Delta\bx_k$
	\EndFor\label{AlgLine:MinNewtonEnd}
	\State \textit{// $\|\bx_{\itersPrimal} - \bx_\infty\|_2 \leq 3 \cdot \rho$}\label{Algo:LineL:PostPrimalIter}\vspace{2mm}
	\State $\cbx := \bx_{\itersPrimal},\quad\cblambda:=-1/\omega\cdot(\bA\cdot\bx_{\itersPrimal}-\bb)$
	\State $\cbmu_L:=\tau_A/(\be+\bx_{\itersPrimal})$,\quad$\cbmu_R:=\tau_A/(\be-\bx_{\itersPrimal})$
	\State $\cbz := (\cbx,\cblambda,\cbmu_L,\cbmu_R)$ \quad \textit{// $\cbz \in \cN_h(\tau_A,\nu_2) \subset \cN(\tau_A,\nu_2)$}\vspace{2mm}
	\State $(\br\hp{1},\br\hp{2},\br\hp{3},\br\hp{4}) := F_\htau(\cbz)$ \Comment{\textit{error-reset step}}
	\State Solve $DF(\cbz) \cdot \Delta\cbz = -(\br\hp{1},\br\hp{2},\bO_n,\bO_n)$ for $\Delta\cbz$\,.\label{AlgLine:Newton2}
	\State $\bz := \cbz + \Delta\cbz$\quad \textit{// $\bz \in \cN(\tau_A,\nu_0)$} \label{AlgLine:EndIniGuess}\label{AlgLine:AugmentationEnd} \label{AlgLine:zInNtauA}\vspace{2mm}
	\State \textit{// - - - Path-following - - -}
	\State $\tau:=\tau_A$
	\For{$k=1,...,\itersPrimalDual$}
		\State $\bz_0 := \bz$,\quad $\htau := \sigma \cdot \tau$\quad \textit{// $\bz_0 \in \cN(\tau,\nu_0) \subset \cN(\tau,\nu_2)$}\vspace{2mm}
		\State Solve $DF(\bz_0) \cdot \Delta\bz_1 = -F_\htau(\bz_0)$ for $\Delta\bz_1$\,.\Comment{\textit{path step}}\label{AlgLine:Newton3a}
		\State $\bz_1 := \bz_0 + \Delta\bz_1$ \quad\quad  \textit{// $\bz_1 \in \cN(\htau,\nu_1) \subset \cN(\htau,\nu_2)$}\vspace{2mm}
		\State Solve $DF(\bz_1) \cdot \Delta\bz_2 = -F_\htau(\bz_1)$ for $\Delta\bz_2$\,.\Comment{\textit{centrality step}}\label{AlgLine:Newton3b}
		\State $\bz_2 := \bz_1 + \Delta\bz_2$ \quad\quad  \textit{// $\bz_2 \in \cN_h(\htau,\nu_2) \subset \cN(\htau,\nu_2)$}\vspace{2mm}
		\State $(\br\hp{1},\br\hp{2},\br\hp{3},\br\hp{4}) := F_\htau(\bz_2)$ \Comment{\textit{error-reset step}}%\textcolor{red}{put comments through tabs to same depth, but without braking lines }
		\State Solve $DF(\bz_2) \cdot \Delta\bz_2 = -(\br\hp{1},\br\hp{2},\bO_n,\bO_n)$ for $\Delta\bz_3$\,.\label{AlgLine:Newton3c}
		\State $\bz_3 := \bz_2 + \Delta\bz_3$ \quad\quad\textit{// $\bz_3 \in \cN(\htau,\nu_0)$}\vspace{2mm}
		\State $\tau:=\htau$,\quad$\bz:= \hbz_3$
		\If{$\htau \leq \tau_E$}
			\State \textbf{break for-loop}
		\EndIf
	\EndFor\vspace{2mm}
	\State \textit{// - - - Termination - - -}
	\State \textit{\textit{// $\bz \equiv (\bx,\blambda,\bmu_L,\bmu_R)\in\cN(\tau,\nu_0)$, where $\tau \leq \tau_E$}}
	\State \Return $\bx$
\EndProcedure
\end{algorithmic}
}
\caption{Interior point Method}\label{Algo:IPM}
\end{algorithm}

\section{Proof of correctness}

We prove that, {in absence of rounding errors,} Algorithm~\ref{Algo:IPM} returns a solution $\bx$ which satisfies \eqref{eqn:SolutionConditions}. To this end we spend one subsection of text for each section of the algorithm.
In order to improve readability of the paper, most of the proofs will be given in appendix A.

We recall that in absence of rounding errors the algorithm simplifies and we do not need to  compute $\Delta\bz_2,\Delta\bz_3$, i.e. we can set $\Delta\bz_2:=\bO$, $\Delta\bz_3:=\bO$\,. Further $\nu_0=\nu_1=\bO$.

We emphasize that in Section~\ref{sec:NumStab} we show the more general result: namely, that Algorithm~\ref{Algo:IPM} returns a solution $\bx$ which satisfies \eqref{eqn:SolutionConditions} even despite numerical rounding errors.

\subsection{Initialization}
We begin with the first part of the first section, i.e. lines 4--9\,. Let us point out the following properties: $f$ is self-concordant. It has the following gradient
\begin{align*}
	\nabla f(\bx) =& \frac{1}{\tau_A} \cdot \Bigg( \bQ \cdot \bx + \omega \cdot \bx + \bc - \bA\t \cdot \frac{-1}{\omega}\cdot(\bA \cdot \bx - \bb) \Bigg)\\
	&\quad- \Bigg( \frac{1}{\be+\bx} - \frac{1}{\be-\bx} \Bigg)
\end{align*}
and the Hessian
\begin{align*}
	\nabla^2 f(\bx) = \underbrace{\frac{1}{\tau_A} \cdot \Big( \bQ + \omega\cdot\bI + \frac{1}{\omega} \cdot \bA\t\cdot\bA\Big)}_{\text{term 1}} +\underbrace{\vphantom{\frac{1}{\tau_A}} \opdiag(\be+\bx)^{-2} + \opdiag(\be-\bx)^{-2}}_{\text{term 2}}\,.
\end{align*}
Due to the large value $\tau_A$ we find that the 2-norm of term 1 is bounded by $4$. If $\bx \in \cB_{0.5}(\bx_0)$ then we obtain the upper bound $6$ for the 2-norm of term 2 because then each diagonal entry lives in the interval $[0.5\,,\,1.5]$ as $\bx_0=\bO$. It follows
\begin{align}
\|\nabla^2 f(\bxi)\|_2 \leq C_{Hf} \quad\quad \forall \bxi \in \cB_{0.5}(\bx_0)\label{eqn:boundsup_hf}
\end{align}
where $C_{Hf}$ is a method parameter. At $\bx=\bO$ we find
\begin{align*}
	\|\nabla f(\bO)\|_2 = \frac{1}{\tau_A} \cdot \left\| \bc - \frac{1}{\omega} \cdot \bA\t\cdot\bb \right\|_2 < 0.25
\end{align*}
due to the large value of $\tau_A$. Due to the logarithmic terms it further holds $\bI \matleq \nabla^2 f(\bxi)\quad \forall \bxi \in {\Omega}$.

Consider the following result.
\begin{boxThm}[Newton's method for minimization]\label{Thm:NewtonMinimization}
	Given a self-concordant function $\MinFun : \dom(\MinFun)\subset \R^d \rightarrow \R$, $\dom(\MinFun)$ open, $\bu_0 \in \dom(\MinFun)$ and $\rho\in(0,1)$. Let $\MinFun$ satisfy $\bI \matleq \nabla^2 \MinFun(\tbu)$ and $\|\nabla^2 \MinFun(\tbu)\|_2\leq C$ for all $\tbu \in \cB_{0.5}(\bu_0)$ and let $\|\nabla \MinFun(\bu_0)\|_2 < 0.25$\,. Let $\bu_\infty$ be the unique global minimizer of $\MinFun$. Define
	\begin{align}
	K := \left\lceil \log_2\big(1-\log_2(\rho/C)\big)  \right\rceil\label{eqn:K}
	\end{align}
	and the sequence $\lbrace \bu_k \rbrace_{k \in \N} \subset \dom(\MinFun)$ recursively as
	\begin{align*}
	\bu_k := \bu_{k-1} - \big(\nabla^2 \MinFun(\bu_{k-1})\big)\inv \cdot \nabla \MinFun(\bu_{k-1})\quad \forall k \in \N\,.
	\end{align*}
	Then it holds
	\begin{subequations}
		\begin{align}
		\MinFun(\bu_k) &\leq \MinFun(\bu_\infty) + \rho & &\forall k \geq K\,,\label{eqn:Thm:NewtonMinimization:1}\\
		\|\nabla \MinFun(\bu_k)\|_2 &\leq \rho & &\forall k \geq K\,,\label{eqn:Thm:NewtonMinimization:2}\\
		\|\bu_0 - \bu_k\|_2 &\leq 0.5 & &\forall k \in \N\cup\lbrace\infty\rbrace\,,\label{eqn:Thm:NewtonMinimization:3}\\
		\|\bu_\infty - \bu_k\|_2 &\leq 2 \cdot \rho & &\forall k \geq K\,.\label{eqn:Thm:NewtonMinimization:4}
		\end{align}
	\end{subequations}
	\begin{proof} Appendix A.\suppressqed\end{proof}
\end{boxThm}
%\todo{make references to proofs consistent everywhere. Use proof template from style file.}

We use this result to analyze the primal iteration in Alg~\ref{Algo:IPM} lines 5--8\,. To this end we insert $\MinFun = f$, $\bu_0 = \bx_0 {=\bO,\bu_\infty=\bx_\infty }$. We obtain the sequence $\bu_k \equiv \bx_k$, $k=0,...,\itersPrimal$\,. We showed above that the requirements of the theorem are satisfied with $C = C_{Hf}$. For $\rho$ we use the value defined in \ConstantsRef. Then, we obtained for $\itersPrimal$ the value defined in \eqref{eqn:K}.

The theorem then guarantees that the proposition in Alg.~\ref{Algo:IPM} line 9 holds true. In fact it shows the sharper bound $\|\bx_K-\bx_\infty\|_2 \leq 2 \cdot \rho$. But since later we will also consider numerical rounding, we added a margin $\rho$ to the bound.
\largeparbreak

The second part of the initialization consists of Alg.~\ref{Algo:IPM} lines 10--12\,. A primal-dual vector $\cbz$ is computed from $\bx_{\itersPrimal}$\,. After substitution, we have $F_{\tau_A}(\cbz) =\big(\,\nabla f(\bx_{\itersPrimal}),\bO,\bO,\bO\,\big)$.

If $\bx_{\itersPrimal} = \bx_\infty$ was true, then obviously $F_{\tau_A}(\cbz) = \bO$ held by construction, i.e. $\cbz \in \cN(\tau_A)$ would follow. But, $\|\bx_{\itersPrimal} - \bx_\infty\|_2 \leq 3 \cdot \rho$ and $\|\bx_\infty\|_2\leq 0.5$ hold true {because of \eqref{eqn:Thm:NewtonMinimization:3}--\eqref{eqn:Thm:NewtonMinimization:4} and because of our choice of $\bx_0$,} from which the distance in 2-norm of $\cbz$ to $\cN(\tau_A)$ can be bounded as stated in the following Lemma, where $\nu_2$ is a method parameter.

\begin{lem}\label{lem:Bound_cbz}
	Let $\bx_{\itersPrimal}$ satisfy $\|\bx_{\itersPrimal}-\bx_{\infty}\|_2\leq 3 \cdot \rho$. Define $\cbz$ according to Alg.~\ref{Algo:IPM} lines~10--12\,. Then it holds:
	\begin{align*}
		\cbz \in \cN_h(\tau_A,\nu_2)
	\end{align*}
\end{lem}
\begin{proof} Appendix A.\suppressqed\end{proof}
\largeparbreak

The final part of the initialization is a primal-dual iteration in lines 13--15, that we refer to as \textit{error-reset step}.
In fact, it is possible to prove that  from $\cbz \in \cN_h(\tau_A,\nu_2)$ this modified Newton step computes a new iterate $\bz \in \cN(\tau_A) \subset \cN(\tau_A,\nu_0)$. This is proved in the subsequent Theorem~\ref{thm:ErrorResetStep} where also a more general result is given.
\subsection{Path-following}
In this section we consider Newton iterations on the primal-dual vector $\bz$. The root-function used within the Newton iteration is $F_\tau$ for suitable values of $\tau$.  In the appendix we provide  technical intermediate results that allow us to prove the following theorem. It states that, starting from a point $\bz\in\cN(\tau)$, one Newton step provides an improved point in $\hbz \in \cN(\htau)$, where $\htau = \sigma \cdot \tau$.

\begin{boxThm}[Path step]\label{thm:Newton-PathStep}
	Let $\bz \in \cN(\tau)$, where $\tau \in [\tau_E,\tau_A]$. Define $\htau := \sigma \cdot \tau$ and solve the linear system
	\begin{align}
		DF(\bz) \cdot \Delta\bz = -F_\htau(\bz)\,.
	\end{align}
	Compute $\hbz := \bz + \Delta\bz$. Then it holds: $\hbz \in \cN(\htau)$\,.
	\begin{proof} Appendix A.\suppressqed\end{proof}
\end{boxThm}

\subsection{Termination}
The focus of this subsection is on showing the following result, that is relevant for the vector $\bz$ in Alg.~\ref{Algo:IPM} line~33\,. $\nu_0,\tau_E$ are method parameters. For ease of presentation we ignore the envelope, i.e. we consider $\cN(\tau)$ instead of $\cN(\tau,\nu_0)$. We come back to the latter at the end of the subsection.
\begin{boxThm}[$\tol$-accurate solution]\label{thm:tol-accurate_Solution}
	Let $\bz \equiv (\bx,\blambda,\bmu_L,\bmu_R) \in \cN(\tau)$, where $0<\tau \leq \tau_E$. Then $\bx$ satisfies \eqref{eqn:SolutionConditions}\,.
\end{boxThm}
The theorem shows that the final iterate of the path-following phase holds sufficiently accurate values for the numerical solution $\bx$.

The proof is given at the end of the section. It follows from results that involve duality and convexity properties of special solutions to \eqref{eqn:mainQP} and of the following problem:
\begin{align*}\tag{\regQP}\label{eq:regQP}
	&\min_{\bx \in \overline{\Omega}} 			& 		&q_\omega(\bx):= \frac{\omega}{2} \cdot \|\bx\|_2^2 + q(\bx) + \frac{1}{2 \cdot \omega} \cdot \|\bA \cdot \bx - \bb\|_2^2
\end{align*}%\todo{Label this problem as (\regQP).}
We call the unique minimizer of this convex problem $\bx_\omega^\star$\,.

The following result shows a relation between $\bx^\star$ and $\bx_\omega^\star$. The result holds in particular due to the small value of the method parameter $\omega$.
\begin{lem}\label{lem:relation_QP_regQP}
	Consider the minimizer $\bx^\star$ of problem \eqref{eqn:mainQP} and the unique minimizer $\bx_\omega^\star$ of \eqref{eq:regQP}. It holds:
	\begin{align*}
		q(\bx^\star_\omega) & \leq q(\bx^\star) + \frac{\tol}{2}\\
		\|\bA\cdot\bx_\omega^\star - \bb\|_2 & \leq \boRes + \frac{\tol}{2}
	\end{align*}
\end{lem}
\begin{proof}
Due to optimality of $\bx^\star_\omega$ for \eqref{eq:regQP} and since both $\bx^\star,\bx^\star_\omega \in \overline{\Omega}$, it holds $q_\omega(\bx^\star_\omega)\leq q_\omega(\bx^\star)$\,. From there both propositions can be shown. Further details are given in Appendix A.
\end{proof}

We then introduce a result that relates the primal-dual vector $\bz$, computed within Alg.~\ref{Algo:IPM}, to $\bx_\omega^\star$\,. In the result below, $C_q$ is a method parameter.
\begin{lem}[$\varepsilon$-optimality]\label{thm:EpsOptimality}
	Let $\bz \equiv(\bx,\blambda,\bmu_L,\bmu_R) \in \cN(\tau)$. Then:
	\begin{align*}
		q_\omega(\bx) \leq q_\omega(\bx_\omega^\star) + \varepsilon\,,\tageq\label{eqn:DefEpsOpt}
	\end{align*}
	where
	\begin{align*}
		\varepsilon = 3 \cdot n \cdot \tau\,.
	\end{align*}
\end{lem}
\begin{proof}The proof makes use of a technical duality result. It is given in appendix A.\suppressqed\end{proof}

{As the following result shows, the distance of $\bx^\star_\omega$ and $\bx$ can be bounded in terms of a value for $\varepsilon$ that satisfies \eqref{eqn:DefEpsOpt}.}
\begin{lem}\label{lem:OptimErr_OptimGap}
	Let $\bx \in \overline{\Omega}$ satisfy $q_\omega(\bx) \leq q_\omega(\bx^\star_\omega) + \varepsilon$ for some $\varepsilon\geq 0$. Then:
	\begin{align*}
		\|\bx^\star_\omega - \bx\|_2 \leq \sqrt{\frac{\varepsilon}{\omega}}
	\end{align*}
\end{lem}
\begin{proof}
The proof involves a technical convexity result. It is shown in Appendix A.
\suppressqed
\end{proof}

By definition of $\tau_E$, for $\tau \leq \tau_E$ and $\varepsilon=3 \cdot n \cdot \tau$ it follows
\begin{align*}
	\max\lbrace\,\|\bA\|_2\,,\,C_q\,\rbrace \cdot \sqrt{\frac{\varepsilon}{\omega}} \leq \frac{\tol}{4}\,,\tageq\label{eqn:bound:AQ:eps_omg}
\end{align*}
where $C_q \geq \max_{\bxi \in \overline\Omega}\|\nabla q(\bxi)\|_2$\,.

Consider $\bz \equiv(\bx,\blambda,\bmu_{L},\bmu_R) \in \cN(\tau)$ for $\tau \leq \tau_E$. Using Lemma~\ref{lem:relation_QP_regQP}, Lemma~\ref{lem:OptimErr_OptimGap} and the bound in \eqref{eqn:bound:AQ:eps_omg}, we find:
\begin{subequations}
\begin{align*}
	q(\bx) &\leq q(\bx^\star_\omega) + C_q \cdot \underbrace{\|\bx_\omega^\star - \bx\|_2}_{\leq \sqrt{\varepsilon/\omega}} \leq q(\bx^\star) + \frac{\tol}{2} \tageq\label{eqn:SolutionCondBound1}\\
	\|\bA\cdot\bx-\bb\|_2 &\leq \|\bA \cdot \bx^\star_\omega-\bb\|_2 + \|\bA\|_2 \cdot \|\bx_\omega^\star - \bx\|_2 \leq \boRes + \frac{\tol}{2} \tageq\label{eqn:SolutionCondBound2}
\end{align*}\label{eqn:SolutionCondBound}%
\end{subequations}
The two above bounds yield the thesis in Theorem~\ref{thm:tol-accurate_Solution}\,.

Finally, we want to come back to the enveloped spaces $\cN(\tau,\nu)$. Let $\tbz\equiv(\tbx,\tblambda,\tbmu_L,\tbmu_R) \in \cN(\tau,\nu_0)$. Then, by definition, there exists a vector $\bz \in \cN(\tau)$ such that $\|\tbz - \bz\|_2 \leq \nu_0$. Using a triangular inequality and \eqref{eqn:SolutionCondBound}, we find
\begin{align*}
	q(\tbx) &\leq q(\bx) + C_q \cdot \underbrace{\|\tbx-\bx\|_2}_{\leq \nu_0} \leq q(\bx^\star) + \frac{\tol}{2} + \underbrace{C_q \cdot \nu_0}_{\leq \tol/2}\,,\\
	\|\bA\cdot\tbx-\bb\|_2 &\leq \|\bA \cdot \bx-\bb\|_2 + \|\bA\|_2 \cdot \|\tbx-\bx\|_2 \leq \boRes + \frac{\tol}{2} + \underbrace{\|\bA\|_2 \cdot \nu_0}_{\leq \tol/2}\,.
\end{align*}
We summarize this in the following lemma, that can be directly applied to the iterate in line 33\,.
\begin{lem}[$\tol$-accurate solution in $\cN(\tau,\nu_0)$]\label{lem:Tol_acc_sol}
	Let $\bz \equiv (\bx,\blambda,\bmu_L,\bmu_R) \in \cN(\tau,\nu_0)$, where $0<\tau \leq \tau_E$. Then $\bx$ satisfies \eqref{eqn:SolutionConditions}\,.
\end{lem}

\section{Numerical stability}\label{sec:NumStab}
In this section we show that Alg.~\ref{Algo:IPM} is numerically stable. This means the following: We use this algorithm on a digital computer with IEEE floating point arithmetic and a unit round-off $\epsMach>0$. If $\epsMach$ is sufficiently small then the returned vector $\bx$ of Alg.~\ref{Algo:IPM} still satisfies \eqref{eqn:SolutionConditions} --- in spite of all rounding errors. As in the last section, we spend one subsection of text for each section of the algorithm. In each subsection we show that the respective commented claims in the algorithm still hold true, even though this time we take numerical rounding errors into account.

\subsection{Initialization}
We start with the first part, i.e. Alg.~\ref{Algo:IPM} lines 4--8\,. Recapturing the results from the previous section, we showed that the exact minimizer $\bx_{\infty}$ of $f$ satisfies
$\|\bx_\infty\|_2 \in \cB_{0.5}(\bx_0)$\,. Further, {in Appendix A, using Lemma~\ref{Lem:Boyd} it is shown in the proof of Theorem~\ref{Thm:NewtonMinimization} in \eqref{eqn:Page23} that
the exact Newton iterates $\bx_k$ satisfy
\begin{align}
	\|\bx_k - \bx_\infty\|_2 \leq 2 \cdot \sqrt{C_{Hf}} \cdot {\vartheta_k} \quad\quad \forall k \in \N_0\,,\label{eqn:BoundPrimalRounding}
\end{align}
where $\vartheta_{k+1}$ obeys to the recursive bound $\vartheta_{k+1}\leq 2 \cdot \vartheta_k^2$\,, provided that
 $\vartheta_0 < 0.25$.

Now we modify the recursion for $\vartheta_{k+1}$ such that the above statement holds true also subject to rounding errors. Consider that there is a perturbation $\delta>0$ in each computed iterate $\bx_k$. Thus, we replace the recursive formula for $\vartheta_k$ by
\begin{align*}
	\vartheta_{k+1} := 2 \cdot \vartheta_k^2 + \delta\,,\tageq\label{eqn:PerturbedVarthetaRecursion}
\end{align*}
which has the purpose that now the bound \eqref{eqn:BoundPrimalRounding} still holds for the $\delta$-perturbed primal iterates (notice $\delta \leq 2 \cdot \sqrt{C_{Hf}} \cdot \delta$). For reasonably small perturbations $\delta$ the requirement $\vartheta_k<0.25$ is still satisfied for all $k \in \N_0$. Using induction, from \eqref{eqn:PerturbedVarthetaRecursion} we further find the bound $\vartheta_k \leq \max\lbrace \, 2 \cdot 2^{-(1+2^k)} \,,\, 2 \cdot \delta \, \rbrace$\,.

In summary, we have shown at this point that $\delta$-perturbed primal iterates $\bx_k$ do still satisfy the bound \eqref{eqn:BoundPrimalRounding}, where $\vartheta_k$ obey to the bound $\vartheta_k \leq \max\lbrace \, 2 \cdot 2^{-(1+2^k)} \,,\, 2 \cdot \delta \, \rbrace$\,.
Noting that the perturbation $\delta$ comes from the computation of the Newton step,
in the next step we establish an upper bound for $\delta$. To this end we make use of the following result, where $C_{Hf}, C_{Df}$ are method parameters.
\begin{boxThm}[Primal stability]\label{thm:PrimalStability}
	Let $\bx \in \R^n$, $\|\bx\|_2 \leq 0.5$\,. Then the linear system
	\begin{align*}
		\nabla^2 f(\bx) \cdot \Delta\bx = - \nabla f(\bx)
	\end{align*}
	is well-posed with condition number bounded from above by $ C_{Hf}$ and $\|\nabla f(\bx)\|_2\leq C_{Df}$\,.
\begin{proof}Appendix B.\suppressqed\end{proof}
\end{boxThm}
This means that the absolute numerical rounding error $\delta \in \cO(C_{Hf} \cdot C_{Df} \cdot \epsMach)$ when computing the primal iterates on an IEEE floating point computer with unit round-off $\epsMach$ and with a stable linear system solver (e.g. Householder's QR decomposition with backward substitution \cite{NumericalStability})\,. As $\vartheta_k \le 2 \cdot \delta$  and  \eqref{eqn:BoundPrimalRounding} holds, in order to guarantee that the statement in
Alg.~\ref{Algo:IPM} line 9 still holds true, the user shall choose $\epsMach$ such that $\delta \leq \rho/\sqrt {C_{Hf}}$\,. Then, the result in Alg.~\ref{Algo:IPM} line 9 still holds true. In particular, it must be
\begin{align*}
	const \cdot \epsMach \leq \frac{\rho}{2 \cdot (C_{Hf})^{1.5} \cdot C_{Df} }\,,\tageq\label{eqn:epsMachBoundPrimal}
\end{align*}
where the constant depends on the particular linear system solver being used, cf. \cite{NumericalStability} for details. Notice that since $\delta\ll 0.9$ the property $\|\bx_k\|_2 \leq 0.5$ still holds for all rounding-affected primal iterates because under exact calcuations it would even hold $\|\bx_k\|_2\leq 0.41$, which is shown in the proof of Theorem~\ref{Thm:NewtonMinimization}, cf. \eqref{eqn:TighterBound_xInf}\,.
\largeparbreak

Now we consider lines 10--12 of Alg.~\ref{Algo:IPM}\,. We are not concerned with the error in $\cbz$ that results from a perturbation $\delta$ in $\bx_{\itersPrimal}$ because the proposition in Alg.~\ref{Algo:IPM} line~9 still holds. But we have to be concerned about perturbations in $\cbz$ that arise due to numerical rounding errors in lines 10--11\,. Proceeding as in the proof of Lemma~\ref{lem:Bound_cbz} we can show that the absolute condition number of $\cbz$ is bounded by
\begin{align*}
\sqrt{N} \cdot \left( 1 + \frac{1}{\omega} \cdot \|\bA\|_2 + 8 \cdot \tau_A\right)\,.
\end{align*}
Thus, $\cbz$ is stable with respect to small perturbations $\epsMach$\,.

\largeparbreak

Finally, there is the error-reset step in lines 13--15\,. Theorem~\ref{thm:Boundedness}
in the next subsection implies numerical stability of this Newton step. In particular, it shows that in Alg.~\ref{Algo:IPM} line 15 it holds $\bz \in \cN(\tau_A,\nu_0)$ subject to a sufficiently small value of $\epsMach$\,.

\subsection{Path-following}

One single iteration of Newton's method we call \textit{step}. The path-following is a loop from lines 19--30 that consists of three kinds of Newton steps.

The key ingredient to understand why the path-following is stable with respect to numerical rounding errors is the fact that the norms of all iterates $\bz$, all updates $\Delta\bz$, all Jacobians $DF(\bz)$ and the norms of their inverses {remain} bounded.
\begin{boxThm}[Primal-dual stability]\label{thm:Boundedness}\label{thm:PrimalDualStability}
	Let $\bz \in \cN(\tau,\nu)$, where $\tau \in [\sigma \cdot \tau_E,\tau_A]$ and $\nu \leq \nu_2$ and $\htau \in \lbrace\sigma \cdot \tau,\tau\rbrace$\,. Then the following hold:
	\begin{align*}
	\|\bz\|_2 &\leq C_z\\
	\|DF(\bz)\|_2 &\leq C_{DF}\\
	\|DF(\bz)\inv\|_2 &\leq C_{DFinv}\\
	\|F_\htau(\bz)\|_2 &\leq C_F
	\end{align*}
\begin{proof}Appendix B.\suppressqed\end{proof}
\end{boxThm}
The boundedness implies that the linear systems are well-conditioned and the Newton steps $\Delta\bz$ are always bounded:
\begin{align*}
\cond_2\big(\, DF(\bz) \,\big) &\leq \cDF \cdot \cDFinv =: \kappa_{DF}\,,\\
\|\Delta\bz\|_2 &\leq \cDFinv \cdot C_F\,,
\end{align*}
whenever $\bz \in \cN(\tau,\nu)$\,.

The theorem has the following practical meaning: When solving the linear systems with a numerically stable method then the numerical rounding errors in the updated vector $\hbz := \bz + \Delta\bz$ are bounded in 2-norm by $\cO(\kappa_{DF} \cdot \cDFinv \cdot C_F \cdot \epsMach)$\,. With the help of this theorem we can show later that the claims in Alg.~\ref{Algo:IPM} lines 21,\,23,\,26 still hold in spite of numerical rounding errors (subject to $\epsMach$ sufficiently small). The claims are proven in detail in the next three threorems.

We briefly discuss the size of $\kappa_{DF}$. Typically, $\cDF,\cDFinv$ and also {$C_F$} are very large. However, their logarithms are in the order of $L$. This implies $\log(\kappa_{DF}) \in \cO(L)$. Background in stability analysis tells us that usually one must choose the machine accuracy $\epsMach$ in the order of the reciprocal of the condition number, this is $\cO(1/\kappa_{DF})$ in our case. It follows $-\log_2(\epsMach) \in \cO(L)$ is required, where $-\log_2(\epsMach)$ is the number of digits that must be stored for each computed number. For small tolerances $\tol>0$ it roughly holds $-\log(\epsMach) \in \cO(-\log(\tol))$, meaning that the number of required digits for computation is identical in order of magnitude to the number of digits that is just needed to write down a sufficiently accurate solution.

\largeparbreak

For our particular numerically stable method we consider three different kinds of Newton steps (A), (B), (C), that we define below.
\paragraph{(A) Path step}
This step acts from $\bz \in \cN(\tau,\nu_0)$, where $\tau \in [\tau_E,\tau_A]$. It computes an updated vector $\hbz \in \cN(\sigma\cdot \tau,\nu_1)$. This step is useful because it reduces the value of $\tau$ in the updated primal-dual vector $\hbz$.
\begin{boxThm}[Path step]\label{thm:Newton-PathStep_stability}
	Let $\tbz \in \cN(\tau,\nu_0)$, where $\tau \in [\tau_E,\tau_A]$. Define $\htau := \sigma \cdot \tau$ and solve the linear system
	\begin{align}
	DF(\tbz) \cdot \Delta\tbz = -F_\htau(\tbz)
	\end{align}
	{on an IEEE machine with $\epsMach>0$ sufficiently small and a stable algorithm}.
	Compute $\thbz := \tbz + \Delta\tbz$. Then it holds: $\thbz \in \cN(\htau,\nu_1)$\,.
\begin{proof} Appendix B.\suppressqed\end{proof}
\end{boxThm}
	
\paragraph{(B) Centrality step}
This step acts from $\bz \in \cN(\htau,\nu_1)$\,, where $\htau \in [\sigma \cdot \tau_E\,,\,\sigma \cdot \tau_A]$. It computes an updated vector $\hbz \in \cN_h(\htau,\nu_2)$. This step is useful because elements of $\cN_h$ live closer to the central path than elements of $\cN$.
\begin{boxThm}[Centrality step]\label{thm:Newton-CentralityStep}
	Let $\tbz \in \cN(\htau,\nu_1)$, where $\htau \in [\sigma \cdot \tau_E\,,\,\sigma\cdot\tau_A]$. Solve the linear system
	\begin{align}
	DF(\tbz) \cdot \Delta\bz = -F_\htau(\tbz)\,,
	\end{align}
	{on an IEEE machine with $\epsMach>0$ sufficiently small and a stable algorithm}.
	Compute $\thbz := \tbz + \Delta\tbz$. Then it holds: $\thbz \in \cN_h(\htau,\nu_2)$\,.
	\begin{proof}Appendix B.\suppressqed\end{proof}
\end{boxThm}

\paragraph{(C) Error-reset step}
This step acts from $\bz \in \cN_h(\htau,\nu_2)$ where $\htau \in [\sigma \cdot \tau_E\,,\,\sigma \cdot \tau_A]$. It is a modified Newton step because the right-hand side in the linear system differs from $-F_\htau(\bz)$. The step computes an updated vector $\hbz \in \cN(\htau,\nu_0)$. The benefit of this step lives in the property that the distance $\nu$ of $\hbz$ from the non-enveloped neighborhood space $\cN(\htau)\equiv \cN(\htau,0)$ is reset to the small method parameter $\nu_0$\,.
\begin{boxThm}[Error-reset step]\label{thm:ErrorResetStep}
	Let $\tbz \in \cN_h(\htau,\nu_2)$, where $\htau \in [\sigma \cdot \tau_E\,,\,\tau_A]$. Solve the linear system
	\begin{align}
	DF(\tbz) \cdot \Delta\tbz = -(\tbr\hp{1},\tbr\hp{2},\bO_n,\bO_n)\,,
	\end{align}
	on an IEEE machine with $\epsMach>0$ sufficiently small and a stable algorithm,
	where $(\tbr\hp{1},\tbr\hp{2},\tbr\hp{3},\tbr\hp{4}) := F_\htau(\tbz)$\,.
	Compute $\thbz := \tbz + \Delta\tbz$. Then it holds $\thbz \in \cN(\htau,\nu_0)$\,.
	
	If there are no rounding errors then instead it holds $\thbz \in \cN(\htau)$.
\begin{proof}Appendix B.\suppressqed\end{proof}
\end{boxThm}

\paragraph{Stable path-following mechanism}
We describe how the three types of Newton steps are utilized in the proposed method.
The path step accomplishes for the actual goal of the path-following section in Alg.~\ref{Algo:IPM}. The centrality step and error-reset step make sure that any rounding errors, which lead to growth in $\nu$, do not accumulate over several iterations. This is achieved by reducing the value of $\nu$ below a threshold method parameter $\nu_0$ within each cycle of the for-loop.

Fig.~\ref{fig:stablepathfollowing} illustrates the path-following strategy as used in Alg.~\ref{Algo:IPM}. We start from an iterate $\bz \in \cN(\tau,\nu_0)$. From there we compute $\bz_1$ with a path-step. The step achieves the reduction $\tau \rightarrow \htau$, while on the other hand the envelope grows $\nu_0 \rightarrow \nu_1$\,. Then, $\bz_2$ is computed with a centrality step. Since $\cN$ is replaced by $\cN_h$, the centrality improves (blue instead of red spaces), but on the other hand the envelope grows further $\nu_1 \rightarrow \nu_2$\,. Finally, an error-reset step is used to compute $\bz_3 \in \cN(\htau,\nu_0)$. From this vector we can proceed with the next cycle of the for-loop.

\begin{figure}
	\centering
	\includegraphics[width=0.9\linewidth]{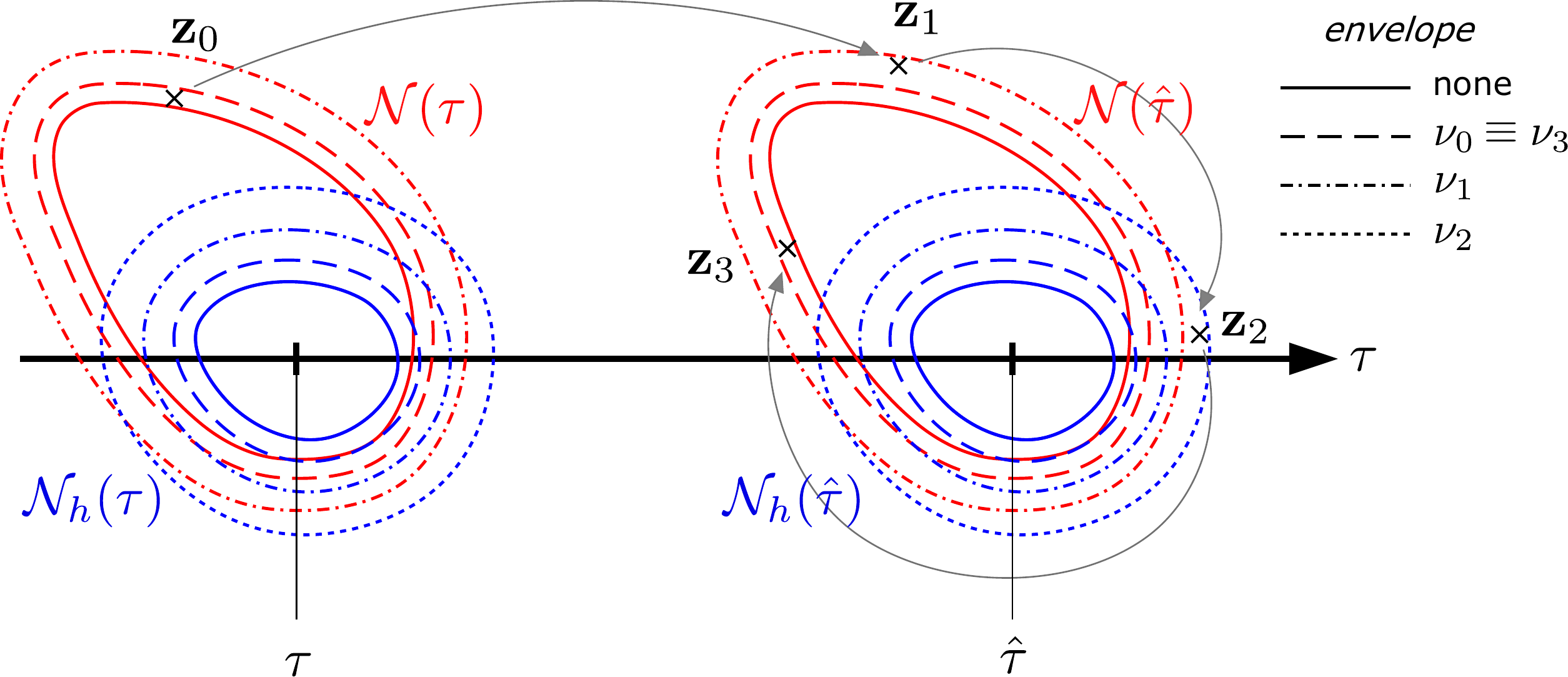}
	\caption{Iterates $\bz_0 \in \cN(\tau,\nu_0)$ to $\bz_3 \in \cN(\htau,\nu_0)$ of stable path-following. The spaces $\cN(\tau,\nu)$ form envelopes of $\cN(\tau)$ for a respective value of $\tau$. From the absolute error bound in each computed iterate for $\bz_0,\bz_1,\bz_2,\bz_3$ one can assure that they live in these enveloped spaces.}
	\label{fig:stablepathfollowing}
\end{figure}

\subsection{Termination}

In the former subsection we showed that despite numerical rounding errors the path-following yields iterates $\bz \in \cN(\tau,\nu_0)$, where eventually $\tau \leq \tau_E$. From there, Lemma~\ref{lem:Tol_acc_sol} shows that Alg.~\ref{Algo:IPM} returns a solution vector $\bx$ that satisfies \eqref{eqn:SolutionConditions}.

\subsection{Sufficiently small value of $\epsMach$}

From \eqref{eqn:epsMachBoundPrimal} and \eqref{eqn:epsMachBoundPrimalDual} we recap the following bounds:
\begin{align*}
	const \cdot \epsMach &\leq \frac{\rho}{2 \cdot C_{Hf} \cdot C_{Df} }\\
	const \cdot \epsMach &\leq \frac{\nu_0}{4 \cdot C_z \cdot \kappa_{DF}}
\end{align*}
The value of $const$ depends only on the linear equation system solver. According to \cite{NumericalStability} this constant is a small integer when using Householder's QR-decomposition with backward substitution. But one can also choose any other linear equation system solver.

All numbers on the right-hand sides live in the open interval $(0,\infty)$, with enumerators $< 1$ and denominators $>1$. Anyway, $|\log(\cdot)|$ of all method parameters are in the order of $L$. It follows $-\log(\epsMach) \in \cO(L)$ is required, where $-\log_{10}(\epsMach)$ is the number of digits that must be stored for each computed number. We discussed formerly that for small tolerances $\tol>0$ the complexity result for $-\log_{10}(\epsMach)$ basically means that the number of required digits for computation is identical in order of magnitude to the number of digits that is just needed to write down a sufficiently accurate solution.

\section{Conclusions}
In this paper we have shown that real-valued convex quadratic programs can be solved in floating-point arithmetic in weakly polynomial time up to an exact tolerance. The proposed method is polynomially efficient and numerically stable, regardless of how ill-conditioned or rank-deficient $\bQ$ and $\bA$ are. All linear systems within the algorithm have bounded condition numbers and all vectors that appear during the computations are bounded. The path-following iteration yields strictly interior iterates despite numerical rounding errors.

This paper had a theoretical focus. We aimed at proving guaranteed success of our method for any real-valued CQP. Practical methods will be limited in the choice that can be made for $\epsMach$ (due to hardware limitations) and $\sigma$. $\sigma$ is the geometric reduction of $\tau$ per iteration. We used a short step method, resulting in $\sigma$ close to $1$. For practical performance it is important to choose $\sigma$ closer to zero, i.e. using a long-step method. Good performance can be achieved, e.g., by choosing $\sigma$ as in Mehrotra's predictor-corrector method, cf. \cite[p.\,411]{Nocedal} and \cite{Mehrotra}. While we have shown numerical stability only for short step methods, further research needs to be conducted to find stability results also for long-step methods.

\FloatBarrier

\bibliography{stableIPM_bib}
\bibliographystyle{plain}

\FloatBarrier

\section{Appendix A: Proofs of correctness section}

\subsection{Proofs of the Initialization section}
\paragraph{Proof of Theorem~\ref{Thm:NewtonMinimization}}
We start summarizing in the following Lemma convergence results of Newton's method for self-concordant functions.
\begin{lem}\label{Lem:Boyd}
	Given $\MinFun : \dom(\MinFun) \subset \R^d \rightarrow \R$ self-concordant and convex, and $\bu_0 \in \dom(\MinFun)$. Define $\MinFun^\star := \min_{\tbu \in \dom(\MinFun)}\lbrace\,\MinFun(\tbu)\,\rbrace$, and the recursive sequences $\lbrace\vartheta_k\rbrace_{k \in \N_0}\subset \R$, $\lbrace\bu_k\rbrace_{k \in \N_0} \subset \dom(f)$ as
	\begin{align*}
	\vartheta_k &:= \sqrt{\nabla \MinFun(\bu_k)\t \cdot \big(\nabla^2 \MinFun(\bu_k)\big)^{-1} \cdot \nabla \MinFun(\bu_k) \,} &  k&=0,1,2,...\\
	\bu_{k+1} &:= \bu_k -  \big(\nabla^2\MinFun(\bu_k)\big)^{-1} \cdot \nabla \MinFun(\bx_k) & k&=0,1,2,...\,.
	\end{align*}
	If $\vartheta_0<0.25$ then it holds $\forall k \in \N_0$:
	\begin{subequations}
		\begin{align}
		\MinFun(\bu_k) &\leq \MinFun^\star + \vartheta_k^2 \label{eqn:BoydGapBound}\\
		{\vartheta_{k+1}} &\leq 2 \cdot \vartheta_k^2 \label{eqn:BoydConvBound}
		\end{align}
	\end{subequations}
\end{lem}
	\begin{proof}Cf. \cite[pp.\,502--505]{Boyd}.\end{proof}
Note that the analysis in \cite[pp.\,502--505]{Boyd} refers to a Newton method with back-tracking. However, therein it is proved that the backtracking line-search accepts the unit step and \eqref{eqn:BoydConvBound} holds whenever  $\vartheta_0<0.25$.

\largeparbreak

Here starts the actual proof of Theorem\,\ref{Thm:NewtonMinimization} where we make use of Lemma \ref{Lem:Boyd}:
First of all, $\bu_0 \in \cB_{0.5}(\bu_0)$ by construction. In the following we show bounds for $\vartheta_k$ under the induction hypothesis that $\bu_j \in \cB_{0.5}(\bu_0)$ $\forall j=0,...,k$\,. Finally, we will show that from the bound of $\vartheta_k$ in turn there follows $\bu_{k+1} \in \cB_{0.5}(\bu_0)$. So \eqref{eqn:Thm:NewtonMinimization:3} is shown by full induction for $k=1,2,3,...\,$.

The updating rule for $\bu_{k}$ and the requirement $\bI \matleq \nabla^2 \MinFun(\bu_{k})$ yield
\begin{align*}
\vartheta_k &= \sqrt{(\bu_{k+1}-\bu_k)\t\cdot \nabla^2\MinFun(\bu_k) \cdot (\bu_{k+1}-\bu_k)\,}\\
&\geq \frac{1}{\sqrt{\Big\|\big(\nabla^2\MinFun(\bu_k)\big)^{-1}\Big\|_2}} \cdot \|\bu_{k+1}-\bu_k\|_2 \geq \|\bu_{k+1}-\bu_k\|_2\,.\tageq \label{eqn:BoundDu}
\end{align*}
Moreover, from $\|\nabla\MinFun(\bu_0)\|_2 < 0.25$ and
\begin{align*}
\vartheta_k &= \sqrt{\nabla\MinFun(\bu_k)\t\cdot \big(\nabla^2\MinFun(\bu_k)\big)^{-1} \cdot \nabla\MinFun(\bu_k)\,} \leq \underbrace{\sqrt{\Big\|\big(\nabla^2\MinFun(\bu_k)\big)^{-1}\Big\|_2}}_{\leq 1} \cdot \|\nabla\MinFun(\bu_k)\|_2
\end{align*}
follows that $\vartheta_0< 0.25$ holds. Thus, all the requirements for Lemma\,\ref{Lem:Boyd} are satisfied.

Then, \eqref{eqn:BoydConvBound} holds and this implies $\vartheta_k \leq 2^{-(1+2^{k})}$. Thus, for $\rho \in (0,1)$ and $C\geq 1$ in Theorem  \ref{Thm:NewtonMinimization}, follows
\begin{align}
C \cdot \vartheta_k \leq \rho\quad \forall k\geq \left\lceil\log_2\big(1+\log_2(C/\rho)\big)\right\rceil\,.\label{eqn:BoundVartheta}
\end{align}
Note that it must be $C\geq 1$ because $\|\nabla^2\MinFun(\bu_k)\|_2\leq C$ and  $\bI\matleq\nabla^2\MinFun(\bu_k)$.
Therefore, proposition \eqref{eqn:Thm:NewtonMinimization:1} follows from \eqref{eqn:BoydGapBound},
and \eqref{eqn:BoundVartheta}.

Using $\|\nabla^2\MinFun(\bu_k)\|_2\leq C$ and
\begin{align*}
& & \vartheta_k \geq& \frac{1}{\sqrt{\|\nabla^2\MinFun(\bu_k)\|_2}} \cdot \|\nabla\MinFun(\bu_k)\|_2
\quad\Rightarrow\quad \|\nabla\MinFun(\bu_k)\|_2 \leq \underbrace{\sqrt{\|\nabla^2\MinFun(\bu_k)\|_2}}_{\leq \sqrt{C} \leq C} \cdot \vartheta_k\,
\end{align*}
shows that proposition \eqref{eqn:Thm:NewtonMinimization:2} holds.

From \eqref{eqn:BoundDu} and $\vartheta_k \leq 2^{-(1+2^{k})}$ we find
\begin{align*}
\|\bu_{j+1}-\bu_{j}\|_2 \leq 2^{-(1+2^j)}\,.
\end{align*}
Proposition \eqref{eqn:Thm:NewtonMinimization:3} now follows from using this bound in the following sum:

\begin{align*}
\|\bu_k-\bu_0\|_2 &\leq \sum_{j=0}^{k-1} \|\bu_{j+1} - \bu_{j}\|_2 \\
&\leq \sum_{j=0}^{k-1} 2^{-(1+2^j)}\leq \sum_{j=0}^{\infty} 2^{-(1+2^j)} \approx 0.408211... <0.5\,.\tageq\label{eqn:TighterBound_xInf}
\end{align*}

Finally we proof proposition \eqref{eqn:Thm:NewtonMinimization:4}. At this point we have already shown $\bu_k \in \cB_{0.5}(\bu_0)$ $\forall k \in \N_0$. We use the following infinite sum:
\begin{align*}
\|\bu_\infty-\bu_k\|_2
&\leq \|\bu_{k+1}-\bu_k\|_2 \cdot \lim\limits_{K\rightarrow\infty}\sum_{j=k}^K \underbrace{\frac{\|\bu_{j+1}-\bu_{j}\|_2}{\|\bu_{k+1}-\bu_{k}\|_2}}_{\leq \sqrt{C} \cdot  4 \cdot 2^{-(1+2^{j-k})}}\\
&\leq \|\bu_{k+1}-\bu_k\|_2 \cdot \underbrace{\sum_{j=0}^\infty \left(\sqrt{C} \cdot 4 \cdot 2^{-(1+2^j)}\right)}_{< \sqrt{C} \cdot 4 \cdot 0.5} < 2 \cdot \sqrt{C} \cdot \vartheta_k\,.	\tageq\label{eqn:Page23}
\end{align*}
In the above formula we use the bound
\begin{align*}
	\frac{\|\bu_{j+1}-\bu_j\|_2}{\|\bu_{k+1}-\bu_k\|_2} \leq \sqrt{C}\cdot\frac{\vartheta_{j}}{\vartheta_k} \leq \sqrt{C} \cdot 4 \cdot 2^{-(1+2^{j-k})}\quad\quad \forall j\geq k \in \N_0\,.
\end{align*}
This bound can be shown by induction in the same way as to show $\vartheta_j \leq 2^{-(1+2^j)}$\, by using \eqref{eqn:BoydConvBound} multiple times in a row.\qed

\paragraph{Proof of Lemma~\ref{lem:Bound_cbz}}
Consider $\bx_\infty$, the exact minimizer of $f$, and $\bx_K$, the approximate minimizer found by Newton's method in lines 4--8\,. Define $\delta\bx := \bx_\itersPrimal - \bx_\infty$\,. {For the vector $\cbz$ in line 12 we have}:
\begin{align*}
	\bz_\infty := \begin{pmatrix}
		\bx_\infty\\
		\frac{-1}{\omega} \cdot (\bA \cdot \bx_\infty - \bb)\\
		\frac{\tau_A}{\be+\bx_\infty}\\
		\frac{\tau_A}{\be-\bx_\infty}
	\end{pmatrix}\,,\quad\quad
 	{\cbz} := \begin{pmatrix}
	\bx_\itersPrimal\\
	\frac{-1}{\omega} \cdot (\bA \cdot \bx_\itersPrimal - \bb)\\
	\frac{\tau_A}{\be+\bx_\itersPrimal}\\
	\frac{\tau_A}{\be-\bx_\itersPrimal}
	\end{pmatrix}
\end{align*}

Since living exactly on the central path {as $F_{\tau_A}(\bz_\infty)=\bO$}, we conclude $\bz_\infty \in \cN_h(\tau_A)$\,. In the following we establish a bound to show $\|\cbz - \bz_\infty\|_2 \leq \nu_2$, which shows the proposition. We have
\begin{align*}
\cbz - \bz_\infty = \begin{pmatrix}
		\delta\bx\\
		\frac{-1}{\omega} \cdot \bA \cdot \delta\bx\\
		\frac{\tau_A}{\be+\bx_{\infty}+\delta\bx}-\frac{\tau_A}{\be+\bx_{\infty}}\\
		\frac{\tau_A}{\be-\bx_{\infty}-\delta\bx}-\frac{\tau_A}{\be-\bx_{\infty}}
		\end{pmatrix}\,.\tageq\label{eqn:Sensitivitiy_cbz}
\end{align*}
For the third and fourth component we can use $\|\bx_\infty\|_2\leq 0.41$, as shown in the proof of Theorem~\ref{Thm:NewtonMinimization}, cf. \eqref{eqn:TighterBound_xInf}. Since {$\|\bx_\infty\|_2<0.41$ and despite numerical rounding $\|\bx_\itersPrimal-\bx_\infty\|_2 \leq 3 \cdot \rho \ll 0.09$}, it follows $\|\bx_K\|_2<0.5$\,. Thus we find
\begin{align*}
	\left\|\frac{\tau_A}{\be+\bx_{\infty}+\delta\bx}-\frac{\tau_A}{\be+\bx_{\infty}}\right\|_\infty &\leq \frac{\tau_A}{0.25} \cdot \|\delta\bx\|_2\,,\\
	\left\|\frac{\tau_A}{\be-\bx_{\infty}-\delta\bx}-\frac{\tau_A}{\be-\bx_{\infty}}\right\|_\infty &\leq \frac{\tau_A}{0.25} \cdot \|\delta\bx\|_2\,.
\end{align*}
Inserting this, we finally get
\begin{align*}
\|\cbz - \bz_\infty\|_2 \leq \sqrt{N}\cdot\left(1+\frac{\|\bA\|_2}{\omega} + 8 \cdot \tau_A \right) \cdot \underbrace{\|\delta\bx\|_2}_{\leq 3 \cdot \rho} \leq \nu_2\,,
\end{align*}
which follows from the way how we chose the method parameter $\rho$.\qed

\subsection{Proofs of the Path-following section}

\paragraph{Proof of Theorem~\ref{thm:Newton-PathStep}}
Our proof and all the following lemmas are adapted {from the KKT system in \cite[eq. 7]{IPM25ylater} to our optimality-function}.

\begin{lem}\label{lem:Numerator}
	{Consider $\beta,\theta$ from \ConstantsRef}. Let $\tau>0$ and $\bz \in \cN(\tau)$.
	Then it holds:
	\begin{align}
	\Bigg\|\begin{pmatrix}
		(\be+\bx)\cdot\bmu_L - \htau\cdot\be\\
		(\be-\bx)\cdot\bmu_R - \htau\cdot\be
		\end{pmatrix}\Bigg\|_2^2 \leq (\theta+\beta)^2 \cdot \tau^2
	\end{align}
\end{lem}

\begin{proof}
		\begin{align*}
		\Bigg\|\begin{pmatrix}
		(\be+\bx)\cdot\bmu_L - \htau\cdot\be\\
		(\be-\bx)\cdot\bmu_R - \htau\cdot\be
		\end{pmatrix}\Bigg\|_2^2
		=
		\Bigg\|\underbrace{\begin{pmatrix}
			(\be+\bx)\cdot\bmu_L - \tau\cdot\be\\
			(\be-\bx)\cdot\bmu_R - \tau\cdot\be
			\end{pmatrix}}_{\bu} + \underbrace{\tau\cdot(1-\sigma)\cdot\begin{pmatrix}
			\be\\
			\be
			\end{pmatrix}}_{\bv}\Bigg\|_2^2
		\end{align*}
		From $\bz \in \cN(\tau)$ follows $\|\bu\|_2 \leq \theta \cdot \tau$. Further, from the method parameter \mbox{$\sigma= 1-\beta/\sqrt{2\cdot n}$} follows $\|\bv\|_2 \leq \beta \cdot \tau$. Altogether we have \[\|\bu+\bv\|_2^2 \leq (\|\bu\|_2 + \|\bv\|_2)^2 \leq (\theta+\beta)^2 \cdot \tau^2\,.\qedhere\]
\end{proof}

\begin{lem}\label{lem:Denominator}
	{Given $\tau>0$ and $\bz \in \cN(\tau)$}, it holds:
	\begin{align}
	\operatornamewithlimits{min}_{1\leq j\leq n}\left\lbrace\, (1+x^{[j]})\cdot\mu_L^{[j]},\,(1-x^{[j]})\cdot\mu_R^{[j]} \,\right\rbrace \geq (1-\theta)\cdot \tau
	\end{align}
\end{lem}

\begin{proof}
	Since $\bz \in \cN(\tau)$ we have
	\begin{align*}
	\left\|\begin{pmatrix}
	(\be+\bx)\cdot\bmu_L - \tau\cdot\be\\
	(\be-\bx)\cdot\bmu_R - \tau\cdot\be
	\end{pmatrix}\right\|_\infty \leq \left\|\begin{pmatrix}
	(\be+\bx)\cdot\bmu_L - \tau\cdot\be\\
	(\be-\bx)\cdot\bmu_R - \tau\cdot\be
	\end{pmatrix}\right\|_2 \leq \theta\cdot\tau\,.
	\end{align*}
	Then, the thesis easily follows from the definition of  the $\infty$-norm.
\end{proof}

\begin{lem}\label{Lem:UVleqU+V}
	Let $\bu,\bv\in\R^d$ such that $\bu\t\cdot\bv\geq 0$. Then: $\|\bu\cdot\bv\|_2\leq0.36 \cdot \|\bu+\bv\|^2_2$.
\end{lem}
\begin{proof}\cite[Lemma 3.3]{IPM25ylater}.\end{proof}

\begin{lem}[Positive scalar-product]
	{Let $\Delta\bz$ as defined in Theorem~\ref{thm:Newton-PathStep}. }
	Then:
	$\Delta\bx\t\cdot(\Delta\bmu_L-\Delta\bmu_R)\geq 0$.
\end{lem}\label{Lem:Positive_scalar-product}

\begin{proof} From the linear equation system for $\Delta\bz$ we find the identity
	\begin{align*}
	\big(\omega\cdot\bI+\bQ+{1/\omega}\cdot\bA\t\cdot\bA\big)\cdot\Delta\bx= \Delta\bmu_L-\Delta\bmu_R\,.
	\end{align*}
	$\phantom{A}$
\end{proof}

Since the Newton step $\Delta\bz=\hbz-\bz$ solves exactly for the root of the linearization of $F_\htau$, we find
\begin{align*}
&F_\htau(\hbz)= \\
&\underbrace{\begin{pmatrix}
(\omega \cdot \bI + \bQ) \cdot (\bx + \Delta\bx) + \bc - \bA\t \cdot (\blambda+\Delta\blambda) -(\bmu_L+\Delta\bmu_L)+(\bmu_R+\Delta\bmu_R)\\
\bA \cdot(\bx+\Delta\bx)-\bb+\omega\cdot(\blambda+\Delta\blambda)\\
\bmu_L\cdot(\be+\bx)+\bmu_L\cdot\Delta\bx+\Delta\bmu_L\cdot(\be+\bx)-\htau\cdot\be\\
\bmu_L\cdot(\be+\bx)+\bmu_L\cdot\Delta\bx+\Delta\bmu_L\cdot(\be+\bx)-\htau\cdot\be
\end{pmatrix}}_{\text{linar part, solved to $\bO$ by the Newton step $\Delta\bz$}}\\
&+ \begin{pmatrix}
\bO_n\\
\bO_m\\
\phantom{-{}}\Delta\bmu_L\cdot\Delta\bx\\
-{}\Delta\bmu_R\cdot\Delta\bx
\end{pmatrix}
=
\begin{pmatrix}
\bO_n\\
\bO_m\\
\phantom{-{}}\Delta\bmu_L\cdot\Delta\bx\\
-{}\Delta\bmu_R\cdot\Delta\bx
\end{pmatrix}
\tageq\label{eqn:PatternF}\,.
\end{align*}
\begin{lem}[bound of $F_\htau(\hbz)$]\label{lem:bound_F_htau_hz}
Let $\tau>0$,  $\bz \in \cN(\tau)$ and $\htau$, $\Delta\bz$ and $\hbz$ as defined in Theorem~\ref{thm:Newton-PathStep}. Then,
	\begin{align}
	\|F_\htau(\hbz)\|_2\leq \theta\cdot\htau\label{eqn:BoundF}
	\end{align}
\end{lem}\label{Lem:BoundF}

\begin{proof} From \eqref{eqn:PatternF} follows
	\begin{align*}
	\|F_\htau(\hbz)\|_2\equiv & \left\|\begin{pmatrix}
	\Delta\bx\cdot\Delta\bmu_L\\
	-\Delta\bx\cdot\Delta\bmu_R
	\end{pmatrix}\right\|_2 \\
	= & \Bigg\|\underbrace{\begin{pmatrix}
		(\be+\bx)^{-1/2} \cdot \bmu_L^{1/2} \cdot \Delta\bx\phantom{(-)}\\
		(\be-\bx)^{-1/2} \cdot \bmu_R^{1/2} \cdot (-\Delta\bx)
		\end{pmatrix}}_{\bu} \ \cdot \ \underbrace{\begin{pmatrix}
		(\be+\bx)^{1/2}\cdot\bmu_L^{-1/2} \cdot \Delta\bmu_L\\
		(\be-\bx)^{1/2}\cdot\bmu_R^{-1/2} \cdot \Delta\bmu_R
		\end{pmatrix}}_{\bv}\Bigg\|_2\,.
	\end{align*}
	{Note that $\Delta\bx\t\cdot(\Delta\bmu_L-\Delta\bmu_R)\geq 0$ by Lemma
		\ref{Lem:Positive_scalar-product}}. Then,
	Lemma\,\ref{Lem:UVleqU+V} with  $\bu\t\cdot\bv\equiv\Delta\bx\t\cdot(\Delta\bmu_L-\Delta\bmu_R)$  leads to:
	\begin{align*}
	&\|F_\htau(\hbz)\|_2\\
	\leq& {}0.36{}\cdot\left\|\begin{pmatrix}
	(\be+\bx)^{-1/2} \cdot \bmu_L^{1/2} \cdot \Delta\bx\phantom{(-)} + (\be+\bx)^{1/2}\cdot\bmu_L^{-1/2} \cdot \Delta\bmu_L\\
	(\be-\bx)^{-1/2} \cdot \bmu_R^{1/2} \cdot (-\Delta\bx) + (\be-\bx)^{1/2}\cdot\bmu_R^{-1/2} \cdot \Delta\bmu_R
	\end{pmatrix}\right\|_2^2\\
	=& {}0.36{}\cdot\left\|\begin{pmatrix}
	(\be+\bx)^{-1/2} \cdot \bmu_L^{-1/2} \cdot \Big(\,(\be+\bx)\cdot\Delta\bmu_L+\bmu_L\cdot\Delta\bx\,\Big)\\
	(\be-\bx)^{-1/2} \cdot \bmu_R^{-1/2} \cdot \Big(\,(\be-\bx)\cdot\Delta\bmu_R-\bmu_R\cdot\Delta\bx\,\Big)
	\end{pmatrix}\right\|_2^2
	\end{align*}
	The Newton step $\Delta\bz$ solves the linearization
	\begin{align*}
	(\be+\bx)\cdot\Delta\bmu_L+\bmu_L\cdot\Delta\bx &=-(\be+\bx)\cdot\bmu_L+\htau\cdot\be\,,\\
	(\be-\bx)\cdot\Delta\bmu_R-\bmu_R\cdot\Delta\bx &=-(\be-\bx)\cdot\bmu_R+\htau\cdot\be\,.
	\end{align*}
	Inserting this into the above yields
	\begin{align*}
	\|F_\htau(\hbz)\|_2 &\leq {}0.36{}\cdot\left\|\begin{pmatrix}
	(\be+\bx)^{-1/2} \cdot \bmu_L^{-1/2} \cdot \Big(\,(\be+\bx)\cdot\bmu_L-\htau\cdot\be\,\Big)\\
	(\be-\bx)^{-1/2} \cdot \bmu_R^{-1/2} \cdot \Big(\,(\be-\bx)\cdot\bmu_R-\htau\cdot\be\,\Big)
	\end{pmatrix}\right\|_2^2\\
	&\leq {}0.36{}\cdot\frac{\left\|\begin{pmatrix}
		(\be+\bx)\cdot\bmu_L-\htau\cdot\be\\
		(\be-\bx)\cdot\bmu_R-\htau\cdot\be
		\end{pmatrix}\right\|_2^2}{\operatornamewithlimits{min}_{1\leq j\leq n}\lbrace\, (1+x^{[j]})\cdot\mu_L^{[j]}\,,\,(1-x^{[j]})\cdot\mu_R^{[j]}\,\rbrace}\,.
	\end{align*}
	Using Lemma\,\ref{lem:Numerator} for the numerator and Lemma\,\ref{lem:Denominator} for the denominator yields
	\begin{align*}
	\|F_\htau(\hbz)\|_2 &\leq {}0.36{}\cdot \frac{(\beta+\theta)^2 \cdot \tau^2}{(1-\theta)\cdot \tau} \leq {\theta}\cdot\sigma\cdot{\tau} = \theta \cdot \htau\,.
	\end{align*}
	{The last inequality holds because the method parameters $\beta,\theta,\sigma$ defined in \ConstantsRef satisfy
		\begin{align*}
		\frac{(\beta+\theta)^2}{1-\theta} \leq \theta \cdot \sigma\,.
		\end{align*}}
	\end{proof}
	
The following result differs from Lemma\,\ref{lem:Denominator} in that it holds for $\hbz$.
\begin{lem}\label{lem:auxInteriorZ}
	Let $\hbz$ as defined in Theorem~\ref{thm:Newton-PathStep}.
	It holds:
	\begin{align}
	\operatornamewithlimits{min}_{1\leq j\leq n}\left\lbrace\, (1+\hat{x}^{[j]})\cdot\hat{\mu}_L^{[j]},\,(1-\hat{x}^{[j]})\cdot\hat{\mu}_R^{[j]} \,\right\rbrace \geq (1-\theta)\cdot \htau\label{eqn:LowerBoundhz}
	\end{align}
\end{lem}

	\begin{proof}Use \eqref{eqn:BoundF} and adapt the proof from Lemma\,\ref{lem:Denominator}.\end{proof}

\largeparbreak

In the following we show that the components of $\hbz$ remain interior.
\begin{lem}[Strict interiorness]\label{Lem:Interiorness}
	Consider $\hbz=(\hbx,\hblambda,\hbmu_L,\hbmu_R)$ as defined in Theorem~\ref{thm:Newton-PathStep}.  It holds $\|\hbx\|_\infty <1$, $\hbmu_L>\bO$, $\hbmu_R>\bO$\,.
\end{lem}

\begin{proof} (by contradiction)
	
	\underline{Case 1: left boundary}\\
	Assume $\exists j \in \lbrace1,...,n\rbrace$ such that at least either $1+\hat{x}^{[j]}\leq 0$ or $\hat{\mu}_L^{[j]}\leq 0$. Due to Lemma\,\ref{lem:auxInteriorZ} it holds that $1+\hat{x}^{[j]}$ and $\hat{\mu}_L^{[j]}$ must have the same sign and differ from zero. Following the initial assumption they are strictly negative. Since $1+{x}^{[j]}$ and ${\mu}_L^{[j]}$ are strictly positive, it must hold
	\begin{align*}
	-\Delta{}x^{[j]} \geq 1+x^{[j]}\quad\text{and}\quad-\Delta{}\mu_L^{[j]} \geq \mu_L^{[j]}\ .
	\end{align*}
	Using \ref{lem:Denominator}, \eqref{eqn:PatternF} and Lemma~\ref{Lem:Interiorness}\, this can be led to contradiction as it holds
	\begin{align*}
	(1-\theta) \cdot \tau \leq (1+x^{[j]}) \cdot \mu_L^{[j]} \leq \Delta{}x^{[j]}\cdot \Delta{}\mu_L^{[j]} \leq \|F_\htau(\hbz)\|_\infty\leq \theta \cdot \sigma \cdot \tau \leq \theta\cdot\tau
	\end{align*}
	and $\theta < 1-\theta$. Thus, $\hbx>-\be$ and $\hbmu_L>\bO$\,.
	
	\underline{Case 2: right boundary}\\
	This case is analogous to Case 1 and shows $\hbx<\be$ and $\hbmu_R>\bO$\,.
\end{proof}

To show Theorem~\ref{thm:Newton-PathStep}, we combine the results:
	\begin{proof}[Proof of Theorem\,\ref{thm:Newton-PathStep}] From \eqref{eqn:PatternF} we know the first two components of $F_\htau(\hbz)$ are zeros. From there, together with Lemma\,\ref{lem:bound_F_htau_hz}, we find $\|(\br\hp{3},\br\hp{4})\|_2 \leq \theta \cdot \htau$. The strict interiorness of $\hbz$ is shown in Lemma\,\ref{Lem:Interiorness}.
\end{proof}

In preparation for proofs of the stability properties of our method, we introduce here the following related minor result.
\begin{cor}[Complementarity refinement]\label{cor:CompRef}
	Given $\bz \in \cN(\tau)$. Compute $\bz_1 := \bz - DF(\bz)\inv\cdot F_\tau(\bz)$. Then $\bz_1 \in \cN_h(\tau)$.
\end{cor}
\begin{proof}
Following the lines of Lemma\,\ref{lem:bound_F_htau_hz} and replacing $\beta$ by $0$ everywhere in the proof, since the value of $\tau$ is unchanged, we get
\begin{align*}
	\|F_\tau(\hbz)\|_2 &\leq {}0.36{}\cdot \frac{\theta^2 \cdot \tau^2}{(1-\theta)\cdot \tau} \le 0.5 \cdot \theta \cdot \tau\,.
	\end{align*}
 This follows from
	\begin{align*}
		\frac{\theta^2}{1-\theta} \leq 0.5 \cdot \theta\,.
	\end{align*}
\end{proof}

\subsection{Proofs of the Termination section}
\begin{proof}[Proof of Lemma~\ref{lem:relation_QP_regQP}]
By optimality of $\bx_\omega^\star$ and due to $\bx^\star \in \overline{\Omega}$ it holds
\begin{align*}
q_\omega(\bx^\star_\omega) &\leq q_\omega(\bx^\star)\\
&\Rightarrow\\
 q(\bx^\star_\omega) + \frac{\omega}{2} \cdot \|\bx_\omega^\star\|_2^2 + \frac{1}{2 \cdot \omega} \cdot \|\bA \cdot \bx^\star_\omega- \bb\|_2^2 &\leq q(\bx^\star) + \frac{\omega}{2} \cdot \|\bx^\star\|_2^2 + \frac{1}{2 \cdot \omega} \cdot \boRes^2 \tageq\label{eqn:Bound_q_xOmega}
\end{align*}
To show the first proposition we omit $\frac{\omega}{2} \cdot \|\bx_\omega^\star\|_2^2\geq 0$ on the left-hand side and cancel $\frac{1}{2 \cdot \omega} \cdot \|\bA \cdot \bx^\star_\omega- \bb\|_2^2$ on the left-hand side against its lower bound $\frac{1}{2 \cdot \omega} \cdot \boRes^2$ on the right-hand side. We obtain
\begin{align*}
q(\bx^\star_\omega) \leq q(\bx^\star) + \underbrace{\frac{\omega}{2} \cdot \|\bx^\star\|_2^2}_{\leq \tol/4}\,,
\end{align*}
{from the definition of $\omega$}.
For the second proposition we subtract $q(\bx_\omega^\star)$ from \eqref{eqn:Bound_q_xOmega} and drop $\frac{\omega}{2} \cdot \|\bx_\omega^\star\|_2^2\geq 0$ on the left-hand side. We arrive at
\begin{align*}
\frac{1}{2 \cdot \omega} \cdot \|\bA \cdot \bx^\star_\omega- \bb\|_2^2 &\leq q(\bx^\star)-q(\bx_\omega^\star) + \frac{\omega}{2} \cdot \|\bx^\star\|_2^2 + \frac{1}{2 \cdot \omega} \cdot \boRes^2\,.
\end{align*}
We can use $q(\bx^\star)-q(\bx_\omega^\star) \leq 2 \cdot \operatornamewithlimits{sup}_{\tbx\in\overline{\Omega}}\lbrace\,|q(\tbx)|\,\rbrace \leq 2 \cdot C_q$ and $\|\bx^\star\|_2\leq \Cx$. Then
\begin{align*}
\frac{1}{2 \cdot \omega} \cdot \|\bA \cdot \bx^\star_\omega- \bb\|_2^2 &\leq 2 \cdot C_q + \frac{\omega}{2} \cdot \Cx^2 + \frac{1}{2 \cdot \omega} \cdot \boRes^2\,.
\end{align*}
Multiplication of this with $2 \cdot \omega\leq2$ and taking the square-root yields
\begin{align*}
\|\bA \cdot \bx_\omega^\star - \bb\|_2 &\leq \sqrt{\chi^2 + \omega \cdot (4 \cdot C_q + n)\,} \leq \chi + \underbrace{\sqrt{\omega} \cdot \sqrt{4 \cdot C_q + n}}_{\leq \tol/4}\,,
\end{align*}
where the bound in the last term holds due to our  choice for the method parameter $\omega$.
\end{proof}

\begin{proof}[Proof of Lemma~\ref{thm:EpsOptimality}]
We show that $\bx$ given by the first $n$ entries of  $\bz \in \cN(\tau)$ for $\tau\leq\tau_E$ is an $\varepsilon$-optimal solution to \eqref{eq:regQP}. To this end we first review a result from duality theory.

Consider the following strictly convex quadratic program %\todo{change equation label below to (\prQP)}
\begin{equation}\tag{\prQP}
	\begin{aligned}
	\min_{\bu \in \R^p} & & \varphi(\bu):= & \frac{1}{2} \cdot \bu\t \cdot \bH \cdot \bu + \bg\t\cdot \bu\\
	\text{subject to} & & \bC \cdot \bu \leq & \bd
	\end{aligned}\label{eqn:QP_P}
\end{equation}
and its dual program %\todo{change equation label below to \duQP}
\begin{equation}\tag{\duQP}
	\begin{aligned}
	\max_{\bv \in \R^q} & & \psi(\bv):= & - \frac{1}{2}\cdot\bv\t\cdot\bC\cdot\bH\inv \cdot \bC\t \cdot \bv\\
	& & & - (\bd + \bC\cdot\bH\inv \cdot \bg)\t\cdot \bv\\
	& & & - \frac{1}{2} \cdot \bg\t\cdot\bH\inv\cdot\bg\\
	\text{subject to} & & \bv \geq & \bO
	\end{aligned}\label{eqn:QP_D}
\end{equation}
for $\bH \in \R^{p \times p}$ symmetric positive definite, $\bg \in \R^p$, $\bC \in \R^{q \times p}$ and $\bd \in \R^q$. We write $\bu^\star$, $\bv^\star$ for optimizers of \eqref{eqn:QP_P}, \eqref{eqn:QP_D}. We call $(\bu,\bv)$ a feasible primal-dual pair if $\bC \cdot \bu\leq \bd$ and $\bv \geq \bO$. We call
\begin{align}
\bH\cdot\bu + \bg + \bC\t \cdot \bv = \bO \label{eqn:GradientCondition}
\end{align}
the gradient condition. We say $\bu$ is $\varepsilon$-optimal for $\varepsilon\geq 0$ if $\bu$ is feasible and
\begin{align*}
\varphi(\bu)\leq \varphi(\bu^\star)+\varepsilon\,.
\end{align*}

\begin{boxThm}[Duality gap bounds optimality gap]\label{Thm:OptimalityGap}
	Given a feasible primal-dual pair $(\bu,\bv)$ that satisfies the gradient condition. Then it holds:
	\begin{subequations}
		\begin{align}
		\varphi(\bu)\geq \varphi(\bu^\star)&=\psi(\bv^\star)\geq \psi(\bv)\,,\label{eqn:ThmStrictDuality}\\
		\varphi(\bu) - \psi(\bv) &= \bv\t \cdot (\bd - \bC \cdot \bu)\,.\label{eqn:OptimalityGap}
		\end{align}
	\end{subequations}
	I.e., $\bu$ is $\varepsilon$-optimal for $\varepsilon=\bv\t \cdot (\bd - \bC \cdot \bu)$.
\end{boxThm}

The KKT conditions of \eqref{eqn:QP_P} are
	\begin{subequations}
		\label{eqn:KKT_P}
		\begin{align}
		\bH \cdot \bu + \bg + \bC\t \cdot \by &= \bO \label{eqn:KKT_P:1} \\
		\bd - \bC \cdot \bu =: \bs &\geq \bO\\
		\by &\geq \bO\\
		\by \cdot \bs &= \bO
		\end{align}
	\end{subequations}
	with Lagrange multipliers $\by \in \R^q$ and slacks $\bs \in \R^q$. The KKT conditions of \eqref{eqn:QP_D} are
	\begin{subequations}
		\label{eqn:KKT_D}
		\begin{align}
		\bC\cdot\bH\inv\cdot\bC\t \cdot \bv + (\bd+\bC\cdot \bH\inv \cdot \bg) -\bw &= \bO\\
		\bw &\geq \bO\\
		\bv &\geq \bO\\
		\bv \cdot \bw &= \bO
		\end{align}
	\end{subequations}
	with Lagrange multipliers $\bw \in \R^q$. Both problems are strictly convex/concave, thus the solutions $(\bu^\star,\bs^\star,\by^\star),(\bv^\star,\bw^\star)$ of both KKT systems are unique. Multiplying \eqref{eqn:KKT_P:1} fom left with $\bC\cdot\bH\inv$ yields
	\begin{align*}
	\bC \cdot \bH\inv \cdot \bC\t\cdot \by + \bC \cdot \bH\inv \cdot \bg + \underbrace{\bC\cdot\bu}_{= \bd - \bs} = \bO\,.
	\end{align*}
	Considering this while comparing \eqref{eqn:KKT_P} with \eqref{eqn:KKT_D}, we find that $\bv^\star = \by^\star$ and $\bs^\star = \bw^\star$ hold due to equivalence of \eqref{eqn:KKT_P} and \eqref{eqn:KKT_D}.
	
	From insertion of \eqref{eqn:GradientCondition} in $\varphi(\bu)$ we find
	\begin{align*}
	\varphi(\bu) =& \frac{1}{2} \cdot (\bv\t\cdot\bC+\bg\t)\cdot\bH\inv\cdot\bH\cdot\bH\inv\cdot(\bg+\bC\t\cdot\bv)\\
	&\quad-\bg\t\cdot\bH\inv\cdot(\bg+\bC\t\cdot\bv)\\
	=&-\frac{1}{2}\cdot\bg\t\cdot\bH\inv\cdot\bg + \frac{1}{2}\cdot\bv\t\cdot\bC\cdot\bH\inv\cdot\bC\t\cdot\bv
	\end{align*}
	and for $\psi(\bv)$ we find
	\begin{align*}
	\psi(\bv) =& -\frac{1}{2}\cdot\bv\t\cdot\bC\cdot\bH\inv\cdot\bC\t\cdot\bv-\bv\t\cdot(\bd-\bC\cdot\bu-\bC\cdot\bH\inv\cdot\bC\t\cdot\bv) \\
	&\quad - \frac{1}{2}\cdot\bg\t\cdot\bH\inv\cdot\bg \\[7pt]
	=& \frac{1}{2}\cdot\bv\t\cdot\bC\cdot\bH\inv\cdot\bC\t\cdot\bv - \frac{1}{2}\cdot\bg\t\cdot\bH\inv\cdot\bg - \bv\t\cdot(\bd-\bC\cdot\bu)\,.
	\end{align*}
	Substraction of the above expressions for $\varphi(\bu)$ and $\psi(\bv)$ shows the second proposition. The first proposition follows easily since $\varphi(\bu)$ and $\psi(\bv)$ are upper and lower bounds for the common optimality value $\varphi(\bu^\star)=\psi(\bv^\star)$. \end{proof}

There holds further the following error bound.
\begin{boxThm}[Optimality gap bounds distance]\label{Thm:DistLeqOptgap}
	Given an $\varepsilon$-optimal solution $\bu$ to \eqref{eqn:QP_P}. Then it holds
	\begin{align}
	\|\underbrace{\bu^\star-\bu}_{\delta\bu}\|_\bH \leq \sqrt{2 \cdot \varepsilon}\,,
	\end{align}
	where $\|\cdot\|_\bH$ is the induced norm of $\bH$.
\end{boxThm}

\begin{proof} Since $\bu^\star$ is optimal, it is $\nabla \varphi(\bu^\star)\t\cdot\delta\bu\geq 0$ for every feasible direction $\delta\bu$. We find:
	\begin{align*}
	\varepsilon \geq \varphi(\bu)-\varphi(\bu^\star) = \underbrace{\nabla\varphi(\bu^\star)\t\cdot\delta\bu}_{\geq 0} + \frac{1}{2}\cdot\underbrace{\delta\bu\t\cdot\nabla^2\varphi(\bu^\star)\cdot\delta\bu}_{\|\delta\bu\|^2_\bH}\geq \frac{1}{2} \cdot \|\delta\bu\|^2_\bH
	\end{align*}
\end{proof}
In analogy, the following bound can be shown for an $\varepsilon$-optimal solution $\bv$ of \eqref{eqn:QP_D}:
\begin{align*}
	\|\bv-\bv^\star\|_\bV \leq \sqrt{2 \cdot \varepsilon}\tageq\label{eqn:DualOptBound}
\end{align*}
where $\bV:= \bC\cdot\bH\inv \cdot \bC\t$. If $\bV$ is not positive definite then $\|\cdot\|_\bV$ means the induced semi-norm.

We can express \eqref{eq:regQP} as \eqref{eqn:QP_P} by using
\begin{subequations}
	\begin{align}
	\bH&= \omega \cdot \bI + \bQ + \frac{1}{\omega} \cdot \bA\t\cdot\bA\,,\quad&\quad
	\bg&= \bc - \frac{1}{\omega} \cdot \bA\t\cdot\bb\,,\\
	\bC&= \begin{bmatrix}
	-{}\bI_n\\
	\phantom{-}{}\bI_n
	\end{bmatrix}\,,\quad&\quad
	\bd&= \begin{pmatrix}
	\be_n\\
	\be_n
	\end{pmatrix}\,.
	\end{align}\label{eqn:DefH}
\end{subequations}
The gradient condition for \eqref{eq:regQP} is then
\begin{align}
\Big(\omega \cdot \bI + \bQ + \frac{1}{\omega} \cdot \bA\t\cdot\bA\Big) \cdot \bu + \bc - \frac{1}{\omega}\cdot \bA\t\cdot\bb + [\,-\bI \ \ \bI\,]\cdot \bv=\bO\,.
\end{align}

\largeparbreak
Now we have everything in hand prove Lemma~\ref{thm:EpsOptimality}\,:
\begin{proof}[Proof of Lemma~\ref{thm:EpsOptimality}] Consider the component $\bx$ of $\bz \in \cN(\tau)$\,. Note that $F_\tau(\bz)=(\bO,\bO,\br\hp{3},\br\hp{4})$ and 
	consider $\bx,\bmu_L,\bmu_R$. We find they satisfy the gradient condition with $\bu = \bx$ and $\bv = (\bmu_L,\bmu_R)$. Using Theorem \ref{Thm:OptimalityGap} we find $\bx$ is an
	$\varepsilon$-optimal solution to \eqref{eq:regQP} with the following bound for $\varepsilon$:
	\begin{align*}
	\varepsilon &= \begin{pmatrix}
	\bmu_L\\
	\bmu_R
	\end{pmatrix}\t \cdot \begin{pmatrix}
	\be + \bx\\
	\be - \bx
	\end{pmatrix}
	\leq \left\|\begin{pmatrix}
	\bmu_L \cdot( \be + \bx)\\
	\bmu_R \cdot( \be - \bx)
	\end{pmatrix} - \tau \cdot \begin{pmatrix}
	\be\\
	\be
	\end{pmatrix}\right\|_1 + \left\|\tau \cdot \begin{pmatrix}
	\be\\
	\be
	\end{pmatrix}\right\|_1 \\
	&\leq 2 \cdot n \cdot \|F_\tau(\bz)\|_2 + 2 \cdot n \cdot \tau \leq 2 \cdot n \cdot (\theta \cdot \tau + \tau)\,.\tageq\label{eqn:epsOptValue}
	\end{align*}
	In summary, at this point we have shown $q_\omega(\bx) \leq q_\omega(\bx^\star_\omega) + \varepsilon$ for the above bound of $\varepsilon$\,.
\end{proof}

We proceed with the proof of Lemma~\ref{lem:OptimErr_OptimGap}.

\begin{proof}[Proof of Lemma~\ref{lem:OptimErr_OptimGap}] This proof follows from Theorem~\ref{Thm:DistLeqOptgap} with the definition of $\bH$ in \eqref{eqn:DefH}. For an $\varepsilon$-optimal solution $\bx$ to \eqref{eq:regQP} we have
\begin{align*}
	\|\bx - \bx_\omega^\star\|^2_2 \cdot \frac{1}{\|\bH\inv\|_2} \leq \|\bx-\bx_{\omega}^\star\|^2_\bH \leq 2 \cdot \varepsilon
\end{align*}
Considering only the outer expressions, multiplying with $\|\bH\inv\|_2$, and using $\|\bH\inv\|_2\leq 1/\omega$, and finally taking the square-root, we arrive at the proposition.
\end{proof}

\section{Appendix B: Proofs of the stability section}

\subsection{Preliminary results related to spaces}
For some of our proofs below we need general results that relate to the neighborhood spaces $\cN,\cN_h$.
\begin{boxThm}[Space boundedness]\label{thm:SpaceBoundedness}
	Let $\bz \in \cN(\tau,\nu)$ for $\tau \leq \tau_A$\,, $\nu\leq \nu_2$\,. Then it holds:
	\begin{align*}
	\|\bz\|_2 \leq C_z\,.
	\end{align*}
\end{boxThm}
\begin{proof}
Consider first $\cbz=(\cbx,\cblambda,\cbmu_{L},\cbmu_{R})\in \cN(\tau)$.
Since $\cbx \in \Omega$, it follows $\|\cbx\|_2 \leq \sqrt{n}$\,. From the second part of $F_\tau$ we find
\begin{align*}
	\cblambda = \frac{-1}{\omega} \cdot (\bA \cdot \cbx - \bb)\,,
\end{align*}
which leads to $\|\cblambda\|_2\leq \frac{1}{\omega} \cdot (\|\bA\|_2\cdot\sqrt{n}+\|\bb\|_2) \leq C_\lambda$\,. Finally, from the first component of $F_\tau$ we find
\begin{align*}
\cbmu_L - \cbmu_R = (\omega\cdot\bI + \bQ) \cdot \cbx + \bc - \bA\t\cdot\cblambda\,,
\end{align*}
from which follows $\|\cbmu_L - \cbmu_R\|_2 \leq (\omega+\|\bQ\|_2)\cdot\sqrt{n}+\|\bc\|_2+\|\bA\t\|_2\cdot C_\lambda \leq C_{\Delta\mu}$\,.

Now, consider
\begin{align*}
	|\cmu_L^{[j]}\cdot(1+\check{x}^{[j]})-\tau|\leq \theta\cdot\tau\,,\quad\quad |\cmu_R^{[j]}\cdot(1-\check{x}^{[j]})-\tau|\leq \theta\cdot\tau\,,
\end{align*}
from which follows that always at least either $\cmu_L^{[j]}$ or $\cmu_R^{[j]}$ is bounded by $(1+\theta) \cdot \tau$\,. Thus $\mu^{[j]}_L \leq \min\lbrace\mu^{[j]}_L,\mu^{[j]}_R\rbrace+|\mu^{[j]}_L-\mu^{[j]}_R|$ and analogously for each $\mu_R^{[j]}$. It follows:
\begin{align*}
\|(\cbmu_{L},\cbmu_R)\|_\infty \leq (1+\theta) \cdot \tau_A + \|\cbmu_L-\cbmu_R\|_2
\end{align*}
Combining the results, so far we have
\begin{align*}
	\|(\cbmu_L,\cbmu_R)\|_2 \leq \sqrt{2 \cdot n} \cdot \big(\, (1+\theta)\cdot\tau_A + C_{\Delta\mu} \,\big) =: C_\mu\,.
\end{align*}
From the above bounds and our definition of the method parameter $C_z$ we find $\|\cbz\|_2 \leq C_z - \nu_2$ $\forall \cbz \in \cN(\tau)$ $\forall \tau \leq \tau_A$. For $\bz \in \cN(\tau,\nu)$ it follows by definition $\|\bz-\cbz\|_2\leq \nu$. Thus, using the bounds for $\cbz$ and its distance to $\bz$, we have shown the proposition.
\end{proof}
\noindent
This result is useful because it says that an iterate close to the central path is always bounded. The result will be used in the proof of Theorem~\ref{thm:PrimalDualStability}\,.

The next result guarantees interiorness of iterates that live in the enveloped neighborhood spaces.
\begin{boxThm}[Space interiorness]\label{thm:SpaceInteriorness}
	Let $\nu \leq C_\nu \cdot \nu_2$ and $\tau \in [\sigma \cdot \tau_E , \tau_A]$\,. Then it holds
	\begin{align*}
	\cN(\tau,\nu) \subset \cF\,,
	\end{align*}
	where
	\begin{align*}
	\cF := \Big\lbrace \bz \in \R^N \, \Big\vert \ \|\bx\|_\infty \leq 1-\cGap\,,\ \ \bmu_L,\bmu_R\geq \cGap \cdot \be \Big\rbrace\,.
	\end{align*}
\end{boxThm}

The space $\cF$ gives a strict measure of interiorness. The result is remarkable because it assures interiorness not only for iterates in the neighborhood of the central path, but also for iterates in the $\nu$-envelope of that neighborhood. This is particularly useful when dealing with numerical rounding errors because rounding-affected iterates can only live in enveloped spaces. Also this result is needed in our proof of Theorem~\ref{thm:PrimalDualStability}\,.

\begin{proof}[Proof of Theorem~\ref{thm:SpaceInteriorness}]
{Let $\bz \in \cN(\tau,\nu)$. Then it holds
$\|\bz-\hbz\|_\infty \leq \nu$  for a $\hbz \in \cN(\tau)$. Thus, $|x^{[j]}-\hat{x}^{[j]}|\leq \nu$,
$|\bmu_L^{[j]}-\hbmu_L^{[j]}| \leq \nu$  and $|\bmu_R^{[j]}-\hbmu_R^{[j]}| \leq \nu$.
Then, proceeding as in the proof of Lemma}\,\ref{lem:Denominator}\,, we have
$\forall\,j \in \lbrace 1,...,n\rbrace$:
	\begin{align*}
	x^{[j]} &\geq -1+\left(\frac{1-\theta}{\mu_L^{[j]}}\cdot\tau-\nu\right)\,,\quad&\quad
	x^{[j]} &\leq 1-\left(\frac{1-\theta}{\mu_R^{[j]}}\cdot\tau-\nu\right)\,,\\
	\mu_L^{[j]} &\geq \left(\frac{1-\theta}{1+x^{[j]}}\cdot\tau-\nu\right)\,,\quad&\quad
	\mu_R^{[j]} &\geq \left(\frac{1-\theta}{1-x^{[j]}}\cdot\tau-\nu\right)\,.
	\end{align*}
	We use Theorem~\ref{thm:SpaceBoundedness} to find the following bound:
	\begin{align*}
		\frac{(1-\theta) \cdot \tau}{\max\lbrace\,1-x^{[j]},1+x^{[j]},\mu_L^{[j]},\mu_R^{[j]}\rbrace}-\nu \geq \frac{1-\theta}{1+C_z} \cdot \sigma \cdot \tau_E-C_\nu \cdot \nu_2 \geq \cGap >0
	\end{align*}
	It holds $\cGap>0$ due to the way how we defined the other method parameters.
	The bound shows the strict interiorness of $\bx,\bmu_L,\bmu_R$\,.
\end{proof}

\subsection{Proof of Initialization section}

\begin{proof}[{Proof of Theorem~\ref{thm:PrimalStability}}]

We formerly showed $\bI \matleq \nabla^2 f(\bxi)$ $\forall \bxi \in 0.5\cdot {\Omega}$, cf. above Theorem~\ref{Thm:NewtonMinimization}\,. Thus, $\left\|\big(\nabla^2 f(\bxi)\big)\inv\right\|_2 \leq 1$\,. Further, $\|\nabla^2 f(\bxi)\|_2 \leq C_{Hf}$\,. Thus,
$$\cond_2\big(\nabla^2 f(\bxi)\big)\leq C_{Hf}\,.$$
From the requirement $\|\bx\|_2 \leq 0.5$ we can further conclude
\begin{align*}
		&\|\nabla f(\bx)\|_2 \\
	\leq& \frac{1}{\tau_A} \cdot \left\| \bQ \cdot \bx + \omega \cdot \bx + \bc + \frac{1}{\omega} \cdot (\bA \cdot \bx - \bb)\right\|_2 + \left\|\frac{1}{\be + \bx}\right\|_2 + \left\|\frac{1}{\be-\bx}\right\|_2\\
	\leq& \frac{1}{\tau_A} \cdot \Big(\omega + \|\bQ\|_2 + \frac{1}{\omega} \cdot (\|\bA\|_2^2+\|\bb\|_2) \Big) + 4 \cdot \sqrt{n} =: C_{Df}\,.\qedhere
\end{align*}
\end{proof}

\subsection{Proofs of Path-following section}

\paragraph{Proof of Theorem~\ref{thm:Boundedness}}

The bound for $\|\bz\|_2$ follows from Theorem~\ref{thm:SpaceInteriorness}\,.

Let us now establish bounds for the norms of the Jacobian matrix $DF$ {\eqref{eqn:DF_KKT}}, its inverse, and $\|F_\htau(\bz)\|_2$. We can express the Jacobian $DF$ as
\begin{align*}
&DF(\bz) \\
=&
\underbrace{
	\begin{bmatrix}
		\phantom{A}\bI_n\phantom{A} 	&  		& 								& 						\\
	  			& \phantom{A}\bI_m\phantom{A} & 								& 						\\
				& 		& \phantom{-}{}\opdiag(\bmu_L) 	& 						\\
				&  		&  								& -{}\opdiag(\bmu_R)	
	\end{bmatrix}
	}_{=:\bM_L(\bz)} \cdot\, \bK(\bz)\,
		\cdot
\underbrace{\begin{bmatrix}
	\bI_n 	&  			& 			& 					\\
			& -{}\bI_m 	& 			& 					\\
			&  			& -{}\bI_n 	& 					\\
			&  			&  			& \phantom{-}\bI_n	
	\end{bmatrix}}_{=:\bM_R}\,,
\end{align*}
where
\begin{align}
	\bK(\bz) = \left[\begin{array}{c|ccc}
	\bQ + \omega \cdot \bI & \bA\t & \bI & \bI\\
	\hline
	\bA & -\omega\cdot\bI & \bO & \bO \\
	\bI & \bO & -\opdiag\left(\frac{\be+\bx}{\bmu_L}\right) & \bO\\
	\bI & \bO & \bO & -\opdiag\left(\frac{\be-\bx}{\bmu_R}\right)
	\end{array}\right]\,.
\end{align}
$\bK$ has a $2 \times 2$ block-structure. The following lemma can be applied to bound the norm and inverse-norm of $\bK$.
\begin{lem}[Inverse-norm of stabilized saddle-matrix]\label{lem:InvNormRegSaddle}
	Let both $\bV \in \R^{a \times a}$, $\bW \in \R^{b \times b}$ be symmetric positive definite. Let $\bG \in \R^{b \times a}$. Consider
	\begin{align*}
	\bM := \begin{bmatrix}
	\bV & \phantom{A}\phantom{--}\bG\t\\
	\bG & \phantom{A}-\bW
	\end{bmatrix}\,.
	\end{align*}
	Then it holds:
	\begin{subequations}
		\begin{align*}
		\|\bM{}^{-1}\|_2 & \leq \max\lbrace\, \|\bV^{-1}\|_2\,,\,\|\bW^{-1}\|_2\,\rbrace\\
		\|\bM\|_2{}\phantom{^{-1}} & \leq \|\bV\|_2+\|\bW\|_2 + 2 \cdot \|\bG\|_2
		\end{align*}
	\end{subequations}
\end{lem}
\begin{proof}
Consider for an arbitrary $\|\bu\|_2>0$ the equation system
\begin{align*}
	\bM \cdot \underbrace{\begin{pmatrix}
		\bu_1\\
		\bu_2
	\end{pmatrix}}_{=:\bu} = \underbrace{\begin{pmatrix}
	\bv_1\\
	\bv_2
	\end{pmatrix}}_{=:\bv}
\end{align*}
and define $\bw := (\bu_1,-\bu_2) \in \R^{(a+b) \times 1}$\,. We find
\begin{align*}
	\bM \cdot \bu = \bv \quad \Rightarrow \quad \bw\t \cdot \bM \cdot \bu = \bw\t \cdot \bv\,.
\end{align*}
This is equivalent to
\begin{align*}
	\underbrace{\bu_1\t\cdot\bV\cdot\bu_1 + \bu_2\t\cdot\bW\cdot\bu_2}_{\textsf{left term}} = \underbrace{\bv_1\t\cdot\bu_1-\bv_2\t\cdot\bu_2}_{\textsf{right term}}\,.
\end{align*}
We use the following bounds:
\begin{align*}
	\textsf{left term}&\geq \frac{\|\bu_1\|_2^2}{\|\bV\inv\|_2}  + \frac{\|\bu_2\|_2^2}{\|\bW\inv\|_2} \geq \frac{1}{\max\lbrace \|\bV\inv\|_2\,,\,\|\bW\inv\|_2 \rbrace} \cdot \|\bu\|_2^2\\
\textsf{right term}&\leq \|\bv\|_2\cdot\|\bu\|_2
\end{align*}
Inserting the bounds for each term and dividing by $\|\bu\|_2$, we end up with
\begin{align*}
   \|\bu\|_2 \leq {\max\lbrace\,\|\bV\inv\|_2\,,\,\|\bW\inv\|_2\,\rbrace} \cdot \|\bv\|_2\,.
\end{align*}
{Then, recalling that  $\bv= \bM \cdot \bu$, we find}
\begin{align*}
{\|\bM\inv\|_2 := \min_{\bu\neq \bO} \frac{\|\bu\|_2}{\|\bM\cdot\bu\|_2}\le \max\lbrace\,\|\bV\inv\|_2\,,\,\|\bW\inv\|_2\,\rbrace.}
\end{align*}
The proof of the second proposition follows from
\begin{align}
	\|\bM \cdot \bu\|_2 \leq \left\|\begin{pmatrix}\|\bV\|_2 \cdot \|\bu_1\|_2 + \|\bG\t\|_2 \cdot \|\bu_2\|_2\\
	\|\bG\|_2 \cdot \|\bu_1\|_2 + \|\bW\|_2 \cdot \|\bu_2\|_2\\
	\end{pmatrix}\right\|_2\,.\label{eqn:NormBoundBlockMatrix}
\end{align}
\end{proof}

We bound $\|\bK(\bz)\inv\|_2$ for $\bz \in \cN(\tau,\nu_2)$ by using Lemma~\ref{lem:InvNormRegSaddle} with $\bV = \bQ + \omega \cdot \bI$ and $\bW$ the $3\times3$ lower right block. From {Theorems~\ref{thm:SpaceBoundedness}, \ref{thm:SpaceInteriorness} and}
\begin{align*}
\|\bV\inv\|_2 \leq \frac{1}{\omega}\,,\quad\|\bW^{-1}\|_2 \leq \max\Big\lbrace\,\frac{1}{\omega}\,,\,\frac{\bmu_L}{\be+\bx}\,,\,\frac{\bmu_R}{\be-\bx}\,\Big\rbrace \leq \max\left\lbrace\frac{1}{\omega},\frac{C_z}{\cGap}\right\rbrace\,.
\end{align*}
follows $\|\bK(\bz)\inv\|_2\leq\max\big\lbrace\frac{1}{\omega}\,,\,\frac{C_z}{\cGap}\big\rbrace$. Combining all bounds, we find
\begin{align*}
	\|DF(\bz)\inv\|_2 &\leq \underbrace{\|\bM_L(\bz)\inv\|_2}_{\leq 1/\cGap} \cdot \|\bK(\bz)\inv\|_2 \cdot \underbrace{\|\bM_R\|_2}_{\equiv 1} \leq \cDFinv\,.
\end{align*}
{Taking into account the structure of $DF$, using again  Lemma~\ref{lem:InvNormRegSaddle} and proceeding as to prove \eqref{eqn:NormBoundBlockMatrix}, we find
\begin{align*}
       \|DF(\bz)\|_2  \leq &\|\bQ\|_2 +2\cdot \omega + 2\cdot \|\bA\|_2 + 2\cdot\|\bI\|_2 \\
	 	 +&\|\opdiag(\bmu_L)\|_2 + {\|\opdiag(\be+\bx)\|_2}
	 +\|\opdiag(\bmu_R)\|_2 + {\|\opdiag(\be-\bx)\|_2}\\
	\leq & \|\bQ\|_2 + 2 \cdot \omega  + 2\cdot \|\bA\|_2 + 4 \cdot \|\bI\|_2 + 4 \cdot \underbrace{\|\opdiag(\bz)\|_2}_{\leq \|\bz\|_\infty \leq \|\bz\|_2}\\
	\leq & 2 \cdot \omega + \|\bQ\|_2 +  2 \cdot \|\bA\|_F + 4 \cdot C_z + 4 \leq \cDF\,.
\end{align*}

Finally, we show that $\|F_\htau(\hbz)\|_2 \leq C_F$ $\forall \hbz \in \cN(\tau,\nu_2)$\,.
This follows from the lemma below taking into account that there exists $\bz \in \cN(\tau)$ such that $\|\delta\bz\|_2=\|\hat \bz -\bz\|\le \nu_2$. The results in this Lemma will be used also in subsequent proofs.

\begin{lem}\label{lem:NewtonRhsInexact}
	Let $\bz,\delta\bz \in \R^N$, where $\bz \in \cN(\tau)$ and $\|\delta\bz\|_2\leq\nu_2$. Then
	\begin{align}
	\|F_\tau(\bz+\delta\bz)\|_2 &\leq \|F_\tau(\bz)\|_2 + \cdF \cdot \|\delta\bz\|_2\\
	\|DF(\bz+\delta\bz)\|_2 &\leq \|DF(\bz)\|_2 + \cdDF \cdot \|\delta\bz\|_2\,,
	\end{align}
	where $\cdF,\cdDF$ are method parameters.
\end{lem}
\begin{proof}
	The thesis easily follows from the definition of $F_\tau$ and  $DF$.
\end{proof}
\noindent
Here ends the proof of Theorem~\ref{thm:Boundedness}\,.

\paragraph{Proofs of Theorems~\ref{thm:Newton-PathStep_stability},~\ref{thm:Newton-CentralityStep},~\ref{thm:ErrorResetStep}}

Theorem~\ref{thm:Boundedness} shows that the primal-dual systems within the Newton steps are well-conditioned in the sense that the condition number and norms are bounded. The following lemma gives a general result on how the solution of bounded well-conditioned linear systems is affected by perturbations.
\begin{lem}\label{lem:LinSysPerturb}
	Let $\bG,\delta\bG \in \R^{d \times d}$, $\bG$ regular, $\bv,\delta\bv \in \R^d$, and $\|\delta\bG\|_2 < 1/\|\bG\inv\|_2$. Then $\bG+\delta\bG$ is regular. Define
	\begin{align*}
	\bu := \bG\inv\cdot\bv\,,\quad\quad\tbu:=(\bG+\delta\bG)\inv\cdot(\bv+\delta\bv)\,.
	\end{align*}
	Then it holds:
	\begin{align*}
	\frac{\|\bu-\tbu\|_2}{\|\bu\|_2} \leq \frac{\kappa_2(\bG)}{1-\kappa_2(\bG)\cdot\frac{\|\delta\bG\|_2}{\|\bG\|_2}} \cdot \left(\frac{\|\delta\bv\|_2}{\|\bv\|_2}+\frac{\|\delta\bG\|_2}{\|\bG\|_2}\right)
	\end{align*}
\end{lem}
\begin{proof}Cf. \cite[eq. (4.3.4)]{TrustRegionMethods}.\end{proof}

We use this lemma to show the following theorem.
\begin{boxThm}\label{thm:PerturbedNewtonStep}
	Let $\bz \in \cN(\tau)$, where $\tau \in [\sigma \cdot \tau_E,\tau_A]$\,. Let $\htau \in [\sigma \cdot \tau,\tau]$. Define
	\begin{align*}
		\hbz := \bz - \underbrace{DF(\bz)\inv \cdot F_\htau(\bz)}_{=:\Delta\bz}\,.
	\end{align*}
	Let $\tbz \in \cB_{\nu}(\bz)$, where $\nu \in [\nu_0/C_\nu,\nu_2]$. Compute
	\begin{align*}
		\hat{\tbz} := \tbz - \underbrace{DF(\tbz)\inv \cdot F_\htau(\tbz)}_{=:\Delta\tbz}
	\end{align*}
	on an IEEE machine with unit roundoff $\epsMach>0$ sufficiently small, and with a numerically stable method for the solution of the linear system. Then it holds:
	$$\|\hat{\tbz}-\hbz\|_2 \leq C_\nu \cdot \nu$$
\end{boxThm}

\begin{proof}
From Lemma~\ref{lem:LinSysPerturb}, {Lemma~\ref{lem:NewtonRhsInexact}} and Theorem~\ref{thm:Boundedness} we find
\begin{align*}
	\frac{\|\Delta\bz-\Delta\tbz\|_2}{\|\Delta\bz\|_2} \leq& \underbrace{\frac{\kappa_{DF}}{1-\kappa_{DF}\cdot \frac{\cdDF\cdot\nu_2}{\|DF(\bz)\|_2}}}_{\leq C_{\delta\Delta z}} \cdot
	\left(
		\frac{\cdF\cdot\nu}{\|F_\htau(\bz)\|_2} + \frac{\cdDF\cdot\nu}{\|DF(\bz)\|_2}
	\right)\\
	&\quad + \kappa_{DF} \cdot \cO(\epsMach)\,.
\end{align*}
The first term in the bound arises because the linear system for $\Delta\tbz$ is different from that for $\Delta\bz$ in so far, that $F_\htau$ and $DF$ are not evaluated at $\bz$ but at $\tbz$. The second term arises because we use a numerically stable method with a unit roundoff $\epsMach$, which introduces an additional perturbation to $\Delta\tbz$ that depends on the condition number and unit roundoff.

The bound
$$\frac{\kappa_{DF}}{1-\kappa_{DF}\cdot \frac{\cdDF\cdot\nu_2}{\|DF(\bz)\|_2}}\le C_{\delta\Delta z}$$
can be shown by using $\|DF(\bz)\|_2\geq \omega$.
Moreover, multiplying the above inequality with $\|\Delta\bz\|_2$ and further using the bounds 
$$\|\Delta\bz\|_2\leq \|F_\htau(\bz)\|_2 \cdot \|DF(\bz)\inv\|_2\quad , \quad \|\Delta\bz\|_2\leq 2 \cdot C_z\,,$$
of which the latter follows from Theorem~\ref{thm:Boundedness} {(because the $2$-norm of the exact iterates $\bz,\hbz$ is bounded by $C_z$)}, we obtain
\begin{align*}
	\|\Delta\bz-\Delta\tbz\|_2 \leq & C_{\delta\Delta z} \cdot
	\Bigg(
		\frac{\cdF\cdot\nu}{\|F_\htau(\bz)\|_2} \cdot \|F_\htau(\bz)\|_2 \cdot \underbrace{\|DF(\bz)\|_2}_{\leq \cDFinv} + \frac{\cdDF\cdot\nu}{\omega} \cdot 2 \cdot C_z
	\Bigg)\\
	&\quad+2 \cdot C_z \cdot \kappa_{DF} \cdot \cO(\epsMach)\\
	\leq& \underbrace{C_{\delta\Delta z} \cdot (\cdF \cdot \cDFinv + 2/\omega \cdot \cdDF \cdot C_z)}_{\leq C_\nu /2} \cdot \nu\\
	&\quad+ \underbrace{2 \cdot C_z \cdot \kappa_{DF} \cdot \cO(\epsMach)}_{\leq 0.5 \cdot \nu_0} \\
	\leq& 0.5 \cdot C_\nu \cdot \nu + 0.5 \cdot C_\nu \cdot \underbrace{\frac{\nu_0}{C_\nu}}_{\leq \nu} \leq C_\nu \cdot \nu \,,
\end{align*}
where $\epsMach$ must be chosen sufficiently small. In particular,
\begin{align*}
	const \cdot \epsMach \leq \frac{\nu_0}{4 \cdot C_z \cdot \kappa_{DF}}\,,\tageq\label{eqn:epsMachBoundPrimalDual}
\end{align*}
where $const$ again depends on the particular linear equation system solver that is used.
\end{proof}

The theorem can be directly used for proving Theorems~\ref{thm:Newton-PathStep_stability} and \ref{thm:Newton-CentralityStep}\,.
\begin{proof}[Proof of Theorem~\ref{thm:Newton-PathStep_stability}] Consider $\bz \in \cN(\tau)$ with $\tau \in [\tau_E,\tau_A]$ so that $\tbz \in \cB_{\nu_0}(\bz)$. Choose $\htau = \sigma \cdot \tau$. Define $\hbz,\hat{\tbz}$ as in Theorem~\ref{thm:PerturbedNewtonStep}. It says $\|\hbz - \hat{\tbz}\|_2 \leq C_\nu \cdot \nu_0 \leq \nu_1$. From Theorem~\ref{thm:Newton-PathStep} we further know $\hbz \in \cN(\htau)$. Thus $\hat{\tbz} \in \cN(\htau,\nu_1)$\,.
\end{proof}

\begin{proof}[Proof of Theorem~\ref{thm:Newton-CentralityStep}] Consider $\bz \in \cN(\tau)$ with $\tau \in [\sigma \cdot \tau_E,\tau_A]$ so that $\tbz \in \cB_{\nu_1}(\bz)$. Choose $\htau = \tau$. Define $\hbz,\hat{\tbz}$ as in Theorem~\ref{thm:PerturbedNewtonStep}. It says $\|\hbz - \hat{\tbz}\|_2 \leq C_\nu \cdot \nu_1 \leq \nu_2$. From Corollary~\ref{cor:CompRef} we further know $\hbz \in \cN_h(\htau)$. Thus $\hat{\tbz} \in \cN_h(\htau,\nu_2)$\,.
\end{proof}

For the proof of Theorem~\ref{thm:ErrorResetStep} we need the following additional result.
\begin{lem}\label{lem:PathErrResetExact}
	Let $\bz \in \cN_h(\tau,\nu_2)$, where $\tau\geq \sigma \cdot \tau_E$\,. Choose $\htau=\tau$ and define
	\begin{align*}
		(\br\hp{1},\br\hp{2},\br\hp{3},\br\hp{4}) &:= F_\tau(\bz)\\
		\hbz &:= \bz \underbrace{- DF(\bz)\inv \cdot (\br\hp{1},\br\hp{2},\bO_n,\bO_n)}_{=:\Delta\bz}\,.
	\end{align*}
	Then $\hbz \in \cN(\tau)$\,.
\end{lem}
\begin{proof} Taking into account the definition of the neighbourhoods $\cN_h(\tau,\nu_2)$,  $\cN(\tau)$ and  the form  of $F_\tau$
 we find $\|(\br\hp{1},\br\hp{2},\bO,\bO)\|_2 \leq \cdF \cdot \nu_2$. Together with Theorem~\ref{thm:Boundedness} it follows
\begin{align*}
	\|\Delta\bz\|_2 &\leq \|DF(\bz)\inv\|_2 \cdot \|(\br\hp{1},\br\hp{2},\bO_n,\bO_n)\|_2 \\
	&\leq \cDFinv \cdot \cdF \cdot \nu_2 \leq (C_\nu - 1) \cdot \nu_2\,.
\end{align*}
Thus, $\hbz \in \cN_h(\tau,C_\nu \cdot \nu_2)$. Theorem~\ref{thm:SpaceInteriorness} shows $\hbz \in \cF$\,, {i.e. $\|\hbx\|_\infty<1$ and $\hbmu_L,\hbmu_R>\bO$}. According to the definition of $\cN_h(\tau,C_\nu \cdot \nu_2)$, there exists $\cbz \in \cN_h(\tau)$ that satisfies $\|\hbz-\cbz\|_2 \leq C_\nu \cdot \nu_2$. From Lemma~\ref{lem:NewtonRhsInexact} follows
\begin{align*}
	\|F_\tau(\hbz)\|_2 \leq \underbrace{\|F_\tau(\cbz)\|_2}_{\leq 0.5\cdot \theta\cdot\tau} + \underbrace{\cdF \cdot C_\nu \cdot \nu_2}_{\leq 0.5 \cdot \theta \cdot \sigma \cdot \tau_E} \leq \theta \cdot \tau\,.
\end{align*}
Also consider $(\hbr\hp{1},\hbr\hp{2},\hbr\hp{3},\hbr\hp{4}):=F_\tau(\hbz)$. Due to linearity of the Newton step it follows $\hbr\hp{1}=\bO_n,\hbr\hp{2}=\bO_m$\,. Altogether, $\hbz \in \cN(\tau)$\,.
\end{proof}

We now show the final result.

\begin{proof}[Proof of Theorem~\ref{thm:ErrorResetStep}] Choose $\tau := \htau$.
{Let $\tbz \in \cN_h(\hat \tau,\nu_2)$ and  $\thbz$ as stated in Theorem~\ref{thm:ErrorResetStep}. For the proof we make use of Lemma~\ref{lem:PathErrResetExact} with
$\bz = \tbz$ and $\hbz$ as stated in the Lemma.
Then, from Lemma~\ref{lem:PathErrResetExact} follows $\hbz \in \cN(\tau)$.}

{Moreover, proceeding as in the proof of Theorem~\ref{thm:PerturbedNewtonStep} we obtain  a relation between the exact vector $\hat{\bz}$ and the computed vector $\hat{\tbz}$. Namely, we prove that $\hat{\tbz} \in \cB_{(C_\nu \cdot \nu_*)}(\hbz)$, where $C_\nu \cdot \nu_* \leq \nu_0$. Thus,  $\hbz \in \cN(\tau)$ yields $\hat{\tbz} \in \cN(\tau,\nu_0)$.}

If there is no rounding error then $\thbz\equiv\hbz \in \cN(\htau)$.
\end{proof}%\todo{I assume q.e.d. shall be replaced by style-file feature.}

\setlist[description]{labelwidth=1.3cm,leftmargin=!}
\section{Appendix C: Method parameters}

\subsection{Parameters and descriptions}
\paragraph{Dimensions}
\begin{description}
	\item[$n $] dimension of primal vector x
	\item[$m$] number of constraints
	\item[$N$] dimension of primal-dual vector z
\end{description}
\paragraph{KKT regularization}
\begin{description}
	\item[$\omega$] quadratic penalty/ strict convexity parameter
\end{description}
\paragraph{Primal iterations}
\begin{description}
	\item[$\rho$] primal iteration residual tolerance
\end{description}
\paragraph{Central path neighborhood}
\begin{description}
	\item[$\theta$] width of neighborhood of central path
	\item[$\beta$] parameter for reduction of complementarity
\end{description}
\paragraph{Central path parametrization}
\begin{description}
	\item[$\tau_A$] starting point of central path
	\item[$\tau_E$] end point of central path
	\item[$\sigma$] geometric factor of path-following iteration
\end{description}
\paragraph{Iteration counts}
\begin{description}
	\item[$K$] number of primal iterations
	\item[$M $] number of primal-dual iterations
\end{description}
\paragraph{Bounds for iterates and gradients/ Hessians}
\begin{description}
	\item[$C_q $] bound of $|q|$ and $\|\nabla q\|_2$ in $\overline{\Omega}$
	\item[$C_{Hf}$] bound of $\|\nabla^2 f\|_2$ in $0.5 \cdot \overline{\Omega}$
	\item[$C_{\lambda}$] bound of 2-norm of $\blambda$
	\item[$C_{\Delta\mu}$] bound of 2-norm of $\bmu_R-\bmu_L$
	\item[$C_{\Delta z}$] bound of 2-norm of Newton step $\Delta\bz$
	\item[$C_{\delta\Delta z}$] bound of 2-norm of rounding error of $\Delta\bz$
	\item[$C_z$] bound of 2-norm of primal-dual iterate $\bz$
\end{description}
\paragraph{Conditioning and bounds of primal-dual KKT system}
\begin{description}
	\item[$C_F$] norm of KKT residual
	\item[$C_{DF}$] norm of Jacobian of KKT residual
	\item[$C_{DFinv}$] norm of inverse of Jacobian of KKT residual
	\item[$C_{\delta{}F}$] sensitivity of norm of KKT residual with respect to primal-dual vector
	\item[$C_{\delta{}DF}$] sensitivity of norm of Jacobian of KKT residual with respect to primal-dual vector
\end{description}
\paragraph{Enveloped neighborhood spaces}
\begin{description}
	\item[$\cGap$] least absolute component-wise distance of $\bx$ to $-\be,+\be$, and of $\bmu_L,\bmu_R$ to $\bO$
	\item[$\nu_0$] envelope of initial numerical iterate
	\item[$\nu_1$] envelope of iterate after path-following step
	\item[$\nu_2$] envelope of iterate after centrality step
	\item[$C_\nu$] geometric rate of growth of envelope due to numerical rounding errors
\end{description}

\subsection{Computation of parameters}
All the method parameters can be computed from $\bQ,\bc,\bA,\bb,\tol$ in the following order.

The inequalities in the formulas below have the following meaning. We mean $:=$, but since there are numerical rounding errors we want to make sure that some values are rounded to the next smaller floating-point value whereas others are rounded to the next larger floating-point value. The inequalities indicate the direction of rounding that must be used in order to yield numerical stability and correctness subject to rounding errors.
\allowdisplaybreaks
\begin{align*}
	\theta &\leq 0.3\\
	\beta &\leq \theta\\
	\sigma &\geq 1 - \frac{\beta}{\sqrt{2\cdot n}}\\
	N 	&= 3 \cdot n + m\,,\\
	C_{Hf} &\geq 10\\
	C_q &\geq \|\bQ\|_2 \cdot n + \|\bc\|_2 \cdot \sqrt{n}\\
	\omega &\leq \min\left\lbrace\,\frac{\tol}{2 \cdot n},\frac{\tol^2}{4 \cdot C_q+n}\cdot\frac{1}{16},1\,\right\rbrace\\
	C_\lambda &\geq \frac{\|\bA\|_2\cdot \sqrt{n} + \|\bb\|_2}{\omega}\\
	C_{\Delta\mu} &\geq (\omega+\|\bQ\|_2) \cdot \sqrt{n} + \|\bc\|_2 + \|\bA\t\|_2 \cdot C_\lambda\\
	\tau_A &\geq \max\left\lbrace\,\frac{\|\bQ+\omega\cdot\bI+1/\omega \cdot \bA\t\cdot\bA\|_2}{4}\,,\,4 \cdot \|\bc-1/\omega \cdot \bA\t\cdot\bb\|_2\,\right\rbrace\\
	\tau_E &\leq \frac{\tol^2 \cdot \omega}{48 \cdot n \cdot \max\big\lbrace\,\|\bA\|_2\,,\,C_q\,\rbrace}\\
	C_\mu &\geq \sqrt{2 \cdot n} \cdot \big( C_{\Delta\mu} + (1+\theta) \cdot \tau_A \big)\\
	C_z &\geq \sqrt{n + C_\lambda^2 + C_\mu^2} + 0.1\\
	\cGap &\leq \frac{1-\theta}{1+C_z} \cdot \sigma \cdot \tau_E \cdot \frac{1}{2}\\
	\cDF &\geq \|\bQ\|_2 + 2 \cdot \omega + 2 \cdot \|\bA\|_F + 4 + 4 \cdot C_z\\
	\cDFinv &\geq \frac{1}{\cGap} \cdot \max\left\lbrace\,\frac{1}{\omega}\,,\,\frac{C_z}{\cGap}\,\right\rbrace\\
	\kappa_{DF} &\geq \cDF \cdot \cDFinv\\
	\cdF &\geq \cDF\\
	\cdDF &\geq 2\\
	C_{\delta\Delta z} &\geq 2 \cdot \kappa_{DF}\\
	\nu_2 &\leq \min\left\lbrace\,0.1\,,\,\frac{\cGap}{C_\nu}\,,\,\frac{\omega}{2 \cdot \cdDF \cdot \kappa_{DF}}\,,\,\frac{\theta \cdot \sigma \cdot \tau_E}{2 \cdot C_\nu \cdot \cdF} \,\right\rbrace\\
	C_\nu &\geq \max\big\lbrace\, 2 \cdot C_{\delta\Delta z} \cdot (\cdF \cdot \cDFinv + 2/\omega \cdot \cdDF \cdot C_z)\,,\\
	&\quad\quad\,1 + \cDFinv \cdot \cdF \,\big\rbrace\\
	\nu_1 &\leq \nu_2 / C_\nu\\
	\nu_0 &\leq \min\left\lbrace\,\nu_1 / C_\nu\,,\,\frac{\tol}{2 \cdot \max\lbrace\|\bA\|_2,C_q\rbrace}\,\right\rbrace\\
	\rho 	&\leq\frac{1}{4 \cdot \sqrt{N}} \cdot \frac{\nu_2}{1/\omega \cdot \|\bA\|_2 + 1 + 8 \cdot \tau_A}\\
	\itersPrimal &\geq \left\lceil \log_2\left( 1 + \log_2(C_{Hf}/\rho) \right) \right\rceil\\
	\itersPrimalDual &\geq \left\lceil\frac{\log(\tau_E)-\log(\tau_A)}{\log(\sigma)}\right\rceil
\end{align*}

\end{document}